\documentclass[preprint,11pt,authoryear]{elsarticle}

\usepackage[utf8]{inputenc}
\usepackage[english]{babel}
\usepackage{amsmath, amssymb, amsfonts, amsthm}
\usepackage{thmtools}
\usepackage{thm-restate}
\usepackage{mathtools}
\usepackage{booktabs}
\usepackage{algorithm}
\usepackage{algpseudocode}
\usepackage[margin=1in]{geometry}
\usepackage{enumitem}
\usepackage{bm}
\usepackage{natbib}
\usepackage{setspace}
\usepackage{graphicx}   % For including images
\usepackage{subcaption} % For the subfigure environment
\usepackage{xcolor}
\usepackage{url}

\usepackage[normalem]{ulem}

\newcommand{\bvec}[1]{\mathbf{#1}}
\newtheorem{lemma}{Lemma}

\theoremstyle{remark}
\newtheorem{remark}{Remark}

\journal{European Journal of Operational Research}

\begin{document}

    \begin{frontmatter}
        \title{Dynamic Driver Allocation Under Latent Demand Regimes: Indexability of a Partially Observed Markov Decision Process}
         
        \author[csu]{Pedro Cesar Lopes Gerum\corref{cor1}}
        \ead{p.lopesgerum@csuohio.edu}
        \author[csu]{Jiong Liu}
        \author[ufms]{Ellen Bernal Cavalheiro}
        \author[csu]{Luiz Felipe Martins}
        \author[ind]{Matteo Giaretti}

        \cortext[cor1]{Corresponding author}
        \affiliation[csu]{Cleveland State University}  \affiliation[ufms]{Federal University of Mato Grosso do Sul}
        \affiliation[ind]{Independent Researcher}
        
\begin{abstract}

Quick-commerce dark stores dispatch e-grocery orders within 15 to 30 minutes, so operators such as Getir, Glovo, and GoPuff must commit drivers before orders arrive. Demand follows a latent regime that persists across hours, while unfulfilled orders spill forward as a compounding backlog. This decision binds whether drivers are employed on fixed shifts or drawn from a gig platform whose incentives are set ahead of the hour, yet existing models do not learn the regime as orders arrive. We formulate the single-store problem as a partially observable Markov decision process in which the firm infers the regime from realized orders. We show that optimal staffing rises with backlog and prove the single-store problem is indexable, a property open for multi-action partially observed problems in general. We then extend the framework to a driver pool shared across stores through a Lagrangian relaxation that decouples the network into per-store subproblems. The result is a two-level allocation policy that prices the value of tracking demand in real time and ranks stores by a provably valid priority index. On 2021 to 2022 data from the European firm SuperGlovo, which staffs full-time drivers to comply with Spain's Rider's Law, belief updating adds 10.9\%, about \$188K across 27 stores, over the same program with the belief held fixed. The gain concentrates in the stores whose regimes are most persistent, observable before deployment. Online allocation reduces to a table lookup and a ranked list with a cutoff price.
\end{abstract}

        
    \end{frontmatter}

\section{Introduction}

Early entrants in online grocery delivery, such as Webvan and Peapod, invested heavily in dedicated warehousing and fleet infrastructure during the late 1990s, yet neither could generate sufficient order density to cover the fixed costs, and both eventually failed or sharply downsized. The COVID-19 pandemic sharpened expectations around speed and spurred a new fulfillment model, the purpose-designed micro-fulfillment center known as the dark store. These revisit the dedicated-facility idea of the early entrants, but on a smaller and more focused scale. They carry a curated assortment of 1,500 to 3,000 SKUs with real-time inventory visible through a mobile application, so customers only order items that are in stock. Their compact footprint keeps fixed costs manageable, and the urban order density enabled by smartphone adoption provides the volume that earlier warehouse models could not reach. 

The result is reliable delivery within 15 to 30 minutes, and a global market projected to reach \$2,159 billion by 2030 at a compound annual growth rate of 25\% \citep{globenewswire2022}. Operators such as Getir and Glovo in Europe, Freshippo (Hema) in China, and GoPuff in North America have shown that the model scales well across diverse urban markets. Yet the speed advantage that distinguishes dark stores from conventional grocery delivery creates a staffing challenge. For instance, a store's order volume is highly sensitive to transient local conditions. On a rainy evening, or during a nearby sporting event, demand can surge well above its usual level for the hours the surge lasts, and such conditions measurably move store and online demand \citep{SteinkerHobergThonemann2017, BadorfHoberg2020}. Each operating hour, a decision-maker must allocate drivers for the coming period before knowing how many orders will materialize and tight delivery windows leave almost no buffer for demand surges. Allocating too many drivers wastes wages during slow periods, while allocating too few leads to delayed or missed deliveries, eroding the fast-delivery value proposition on which the model depends.

This capacity is determined before demand is known and cannot be revised within the hour. Operators such as Getir, Flink, and Gorillas staff their dark stores with employed riders on fixed shifts \citep{handelsblatt_qcommerce_labor}, and a widening body of regulation is making employment the default. This includes  France's 2020 reclassification of platform drivers, Spain's 2021 ``Rider's Law'', and the European Union's 2024 Platform Work Directive, which presumes an employment relationship and must enter national law across the Union by December 2026 \citep{eu_platform_work_directive_2024}. Glovo, for instance, moved its dark-store grocery couriers onto an employed full-time contracts after fines exceeding 79 million euros for treating more than 10,000 of them as self-employed \citep{bhrrc_glovo_employees}. The same timing binds a gig operator, which commits its wage or per-order incentive ahead of the period, since couriers cannot be summoned the instant a surge begins. Current practice and existing models typically treat this as a repeated single-period optimization, deciding each hour in isolation. However, this approach ignores two features of real demand.

First, demand exhibits within-day persistence. A store experiencing above-average order volume in one hour is more likely to remain in an elevated state the next hour, because the underlying cause (a local event, weather shift, or promotional activity) tends to persist across multiple hours. Observed demand therefore carries information about future demand that a static model discards. The active state is not announced in advance. A scheduled match or holiday could in principle be written into a calendar, but its magnitude at any given store is uncertain, and many drivers of a surge reveal themselves only as orders arrive. The firm must therefore infer the prevailing regime from realized demand in real time.

Second, when delivery capacity falls short of demand, unfulfilled orders accumulate as backlog, creating compounding service degradation in subsequent hours, as the tight delivery promise stretches under load and late orders cascade forward \citep{medianama_10min_rollback}. Customers rank on-time delivery above outright speed, and a single late order measurably raises the chance they abandon the retailer \citep{mckinsey_ecommerce_delivery}.

Moreover, several dark stores within a city sit minutes apart and draw on one shared roster of drivers. Reassigning a driver between them is a short repositioning rather than a long move  \citep{AfecheLiuMaglaras2023}, and thus allocating an additional driver to one store reduces availability for another. When demand varies across stores in real time, the platform must decide each hour which stores most need the shared capacity. This turns a set of independent single-store decisions into a coupled resource-allocation problem, and it calls for a priority rule on how to divide the pool.

We develop a two-level framework for quick-commerce dynamic driver allocation. At the single-store level, we formulate a Partially Observable Markov Decision Process (POMDP) in which the active demand regime is latent and must be inferred from observed orders. We establish that the optimal policy increases monotonically with backlog, and prove indexability of the single-store problem. At the multi-store level, we introduce a Lagrangian relaxation that decouples the network into independent per-store subproblems, each penalized by a shadow price on driver usage. Indexability guarantees that the shadow price defines a well-ordered priority index for each marginal driver.

We make two contributions. The first is a single-store POMDP for quick-commerce driver allocation, in which the demand regime is latent and the firm learns it from realized orders. Unfulfilled orders carry forward as backlog, so a single-period rule that ignores carryover understaffs when demand is high and lets the shortfall compound. The firm also holds a belief over the active demand regime and updates it each hour as demand arrives. Across 27 stores, updating this belief in real time adds 10.9\% ($\approx$ \$188K), over the same dynamic program with the belief held fixed. The gain is not spread evenly, reaching 49\% of a store's reward in those where the regimes are most persistent. Because persistence is visible before deployment, a firm can tell in advance which stores are worth tracking. 

The second contribution is a proof that the single-store problem is \emph{indexable}. Indexability is open for general multi-action partially observed problems. We establish it here for the subclass in which the demand regime evolves independently of the action, the natural structure when staffing serves demand but does not create it. As the shadow price on a driver rises, we show that the optimal number of drivers never rises, so each marginal driver has a well-defined priority index. We further show that this index is non-decreasing in the store’s backlog, so congestion raises a store’s claim on the shared pool. The greedy rule that pools the drivers and fills from the top of the ranking down then coincides with the optimal solution of the Lagrangian relaxation of the network problem. The inference runs offline, so at runtime the policy is a table lookup for a single store and a ranked list with a cutoff price for the shared pool.

The remainder of the paper is organized as follows. Section~\ref{sec:literature} reviews related literature. Section~\ref{sec:problem} describes the operational setting and data. Section~\ref{sec:pomdp} formulates the single-store POMDP.  Section~\ref{sec:demand_estimation} presents the demand regime estimation pipeline. Section~\ref{sec:single_store_results} evaluates the single-store policy against benchmarks and examines the structural properties of the optimal policy. Section~\ref{sec:multistore} extends the framework to multi-store driver allocation via Lagrangian relaxation and index policies, and illustrates the resulting policy on a shared fleet. Section~\ref{sec:conclusions} concludes.

%%=========================================================================
\section{Literature Review}
\label{sec:literature}
%%=========================================================================

We draw on four streams of prior work, on-demand and last-mile delivery operations, demand learning under partial observability, restless bandits and index policies for weakly coupled systems, and staffing under uncertainty. We position our two contributions against the last three.

On-demand and last-mile delivery has been studied primarily as dynamic dispatch and routing. A large body of work models the operator's decisions as a Markov decision process with a fully observed state. These include dynamic same-day delivery with preemptive depot returns \citep{UlmerThomasMattfeld2019}, the restaurant meal delivery problem with random ready times \citep{UlmerThomasCampbellWoyak2021}, order dispatch under multiple sources of uncertainty \citep{LIANG2026100240}, and crowdsourced pickup-and-delivery with ad hoc drivers \citep{ArslanAgatzKroonZuidwijk2019}. A parallel line studies fleet sizing and crowdsourced or hybrid supply, including committed and ad hoc fleets \citep{BehrendtBehrendtWang2024}, service and capacity planning with self-scheduling couriers \citep{YildizSavelsbergh2019}, workforce planning in O2O systems \citep{DaiLiu2020}, and arrival-time estimation for occasional drivers \citep{ZehtabianLarsenWohlk2022}. \citet{SavelsberghUlmer2024} survey the field. The economics of self-scheduling platform labor and capacity is studied by \citet{CachonDanielsLobel2017} and \citet{Taylor2018}. Across this work, the demand state is observed or forecast from history, and the dispatch or staffing decision does not learn a latent demand regime online. We are not aware of any prior operations study of quick-commerce dark-store staffing specifically, where curated assortments and tight 15-to-30-minute promises make the hourly commit-before-demand decision distinctive.

Our first contribution, learning a latent demand regime from realized orders, sits in the literature on operations under imperfectly observed demand. Bayesian demand learning dates to \citet{Scarf1959} and \citet{Azoury1985}, and \citet{LariviereePorteus1999} study Bayesian updating when lost sales go unobserved. When demand is driven by a modulating Markov state, the relevant policies are state-dependent. \citet{SongZipkin1993} characterize optimal base-stock policies under Markov-modulated demand when the state is observed, and \citet{treharne2002adaptive} treat the case where the state is hidden, and the firm carries a belief over it. \citet{ArifogluOzekici2010} derive optimal policies for a finite-capacity inventory system under partially observed Markov-modulated demand and supply. This setting is a POMDP, whose value function is piecewise linear and convex in the belief \citep{sondik1978optimal} and admits structured, often monotone, optimal policies \citep{Lovejoy1987}. We follow this lineage but change the decision. Rather than ordering a single product, the firm staffs a fleet under a backlog that compounds, and our structural results characterize how optimal staffing rises with backlog.

Our second contribution, allocating a shared driver pool by a priority index, builds on restless multi-armed bandits. \citet{whittle1988restless} introduced Whittle's indices, and \citet{weber1990index} established its asymptotic optimality in the binary-action case, while optimal control of the unrelaxed problem is intractable in general \citep{PapadimitriouTsitsiklis1999}. Indexability, the property that makes the index well defined, is delicate and must be verified problem by problem \citep{NinoMora2001}, and is harder still with multiple actions \citep{HodgeGlazebrook2015} or partial observation. Recent work establishes indexability for partially observed scheduling problems \citep{Wang2026} and applies restless-bandit index policies to on-demand resource allocation \citep{Fu2025}. The standard route to a scalable policy relaxes the coupling constraint with a Lagrange multiplier and decomposes the network into per-arm subproblems \citep{Hawkins2003}. Closest to our setting, \citet{MeshramKaza2025} study multi-action partially observed restless bandits and obtain Lagrangian bounds and heuristic policies, noting that index policies face inherent limitations once multiple actions and latent states are present, and indexability for this class remains open in general. In their setting the action also drives the latent state. In ours the demand regime evolves independently of the staffing decision, and this separation is what makes an exact proof possible. \citet{Liu2025relaxed} obtains near-optimal index policies for partially observable restless bandits through a relaxed notion of indexability, whereas we prove exact indexability for our subclass. We prove it for for the single-store problem directly, so the resulting index policy is the exact optimum of the Lagrangian relaxation.

Committing capacity before uncertain demand is a classical concern in service staffing \citep{GansKooleMandelbaum2003} and in the newsvendor problem. We bring these threads together for quick commerce. We model an operator that infers a persistent latent demand regime online to staff a single dark store under compounding backlog, and that allocates a shared fleet across stores by a provably valid priority index. To our knowledge, this is the first treatment to combine online demand-regime learning with a proven indexable allocation policy in this setting.

\section{Problem Description}
\label{sec:problem}
%%=========================================================================

Each operating hour, a decision-maker must decide how many drivers to allocate before customer orders materialize. The decision is made under three interacting sources of uncertainty. First, hourly demand is stochastic and exhibits within-day persistence, i.e., an hour of above-average orders signals that the underlying demand regime is likely elevated. Second, unfulfilled orders spill over as backlog into subsequent hours, so a staffing shortfall compounds over time. Third, the latent demand regime is not directly observable, and the decision-maker must infer it from realized orders through Bayesian updating. In a network of stores sharing a common driver pool, the allocation problem is further complicated by a coupling constraint. Drivers assigned to one store are unavailable to others. We formalize these dynamics as a POMDP at the single-store level and extend it to multi-store allocation via Lagrangian relaxation and index policies. The methodology used is illustrated in Figure \ref{figure:Flowchart_POMDP}.

%\begin{figure}[htbp]
%    \centering
%    \includegraphics[width=0.9\textwidth]{figures/flowchart_POMDP.pdf}
%    \caption{Flowchart POMDP.}
%    \label{figure:Flowchart_POMDP}
%\end{figure}

\begin{figure}[htbp]
    \centering
    \includegraphics[width=0.9\textwidth]{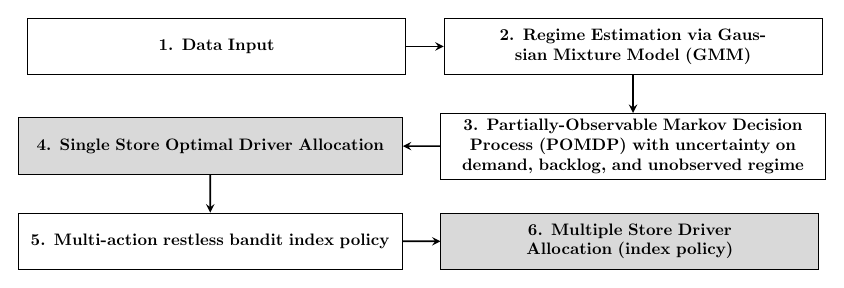}
    \caption{Driver Allocation Flowchart.}
    \label{figure:Flowchart_POMDP}
\end{figure}

%%=========================================================================
\section{Single-Store Formulation}
\label{sec:pomdp}
%%=========================================================================

We model the single-store driver allocation problem as a finite-horizon POMDP. The observable state for each period consists of the backlog $s_t$ (unfulfilled orders carried from previous periods) and the calendar context $(d_t, h_t)$ (day of week and hour). The hidden state is the active demand regime, which the decision-maker cannot observe directly. A regime is a persistent latent state that indexes demand intensity: intuitively, whether the store is experiencing a ``busy'' or ``slow'' period. Each regime $k \in \{1,\ldots,K\}$ is governed by a transition matrix $\bvec{T}$ and determines the distribution of incoming orders. Conditional on regime $k$ being active, the number of orders $x_t$ in period $t$ is drawn from a time-specific distribution $f_k(x;\,d,h)$, called the \emph{emission distribution} (by analogy with hidden Markov models). We treat the emission $f_k$ as a model primitive here and defer its estimation to Section \ref{sec:estimation_overview}.

Because this state is hidden, the decision-maker maintains a $K$-dimensional belief vector $\bvec{b}_t \in \Delta^{K-1}$ (a probability distribution over the $K$ regimes), updated each period via Bayesian inference after observing realized demand $x_t$. The action $a_t \in \{a_{\min}, \ldots, a_{\max}\}$ is the number of drivers to allocate, chosen before demand realizes. The reward captures revenue from fulfilled orders net of driver wages, picking costs, and backlog costs. Transitions govern both the deterministic backlog dynamics and the stochastic belief evolution.

\subsection{Reward}

An order produces revenue $p$ upon fulfillment. Each driver costs $c_a$ per hour in wages and can fulfill up to $v$ orders per hour, so a fleet of $a_t$ drivers can serve at most $v a_t$ orders. Demand exceeding $v a_t$ spills into the next period as backlog. Each fulfilled order also incurs a picking cost $c_p$, reflecting picker labor. Backlogged orders from prior periods are penalized at $c_b$ per order per period, reflecting customer dissatisfaction. We assume positive margins ($v p - v c_p - c_a > 0$). Table \ref{tab:reward_params} summarizes the parameters.

Given effective demand $e_t = x_t + s_t$, observed demand $x_t$, backlog $s_t$, and action $a_t$, the single-period reward is:
\begin{align}
    r(x_t, s_t, a_t) &= \underbrace{p \min\{e_t,\, va_t\}}_{\text{revenue}}
                       - \underbrace{c_a \, a_t}_{\text{driver wages}}
                       - \underbrace{c_p \min\{e_t,\, va_t\}}_{\text{picking cost}}
           - \underbrace{c_b \cdot s_t}_{\text{backlog penalty}},
    \label{eq:reward_spillover}
\end{align}

\begin{table}[h]
\centering\small
\caption{Reward function parameters.}
\label{tab:reward_params}
\begin{tabular}{@{}ll@{}}
\toprule
\textbf{Symbol} & \textbf{Description} \\
\midrule
$p$ & Revenue per fulfilled order \\
$c_a$ & Driver wage per hour \\
$c_p$ & Picking cost per fulfilled order \\
$c_b$ & Backlog penalty per carried-over order per period \\
$c_{\text{lost}}$ & Backlog penalty at the last period \\
$v$ & Maximum orders per driver per hour (capacity) \\
\bottomrule
\end{tabular}
\end{table}

All cost parameters are strictly positive ($c_a, c_p, c_b, c_{\text{lost}}, p > 0$), per-driver capacity satisfies $v \in \mathbb{Z}_{>0}$, demand is integer-valued ($x_t \in \mathbb{Z}_{\ge 0}$). For tractability, we assume the action set is the integer interval $\mathcal{A} = \{a_{\min}, a_{\min}+1, \ldots, a_{\max}\}$. We set $a_{\max}$ large enough that the constraint $a \le a_{\max}$ is non-binding at any backlog level the analysis visits.

Finally, we assume $c_b > q$, where $q \coloneqq p -c_p$ is the net per-order margin, so a carried-over order costs more than the net margin a fulfilled order would generate, because the failure erodes the customer relationship beyond a single transaction. The system is then reactive, with the optimal allocation rising as backlog accumulates.

At the close of the operating day, leftover backlog is not deferred to the next day's operations. Orders unfulfilled by the end of the final period are typically treated as cancellations or customer abandonments rather than rolled forward. We capture this through a terminal lost-sale cost $c_{\mathrm{lost}} > c_b$ per unit of leftover backlog, applied once at the end of the horizon. The asymmetry $c_{\mathrm{lost}} > c_b$ separates the within-day cost of deferring an order one additional period (priced at $c_b$) from the cost of a permanent service failure and rules out terminal-hoarding strategies that would defer backlog clearing to the close of the horizon.

\subsection{Transitions}

The backlog evolves deterministically given realized demand and action:
\begin{equation}
    s_{t+1} = \max\{x_t + s_t - v \, a_t, \; 0\} = \max\{e_t - v \, a_t, \; 0\}.
    \label{eq:backlog_transition}
\end{equation}

The demand, however, is stochastic, and the distribution is unknown (although the set of possible regimes and their corresponding emission distributions is known). All that is observed is the one instance of demand for the period. Therefore, at the start of period $t$, the firm holds belief $\bvec{b}_t$, a prior over which demand regime is currently active. The firm chooses action $a_t$, then demand $x_t$ realizes at calendar time $(d_t, h_t)$. After observing $x_t$, the belief is updated in two steps: \emph{correct} (incorporate the observation) then \emph{predict} (project forward to the next period). This yields $\bvec{b}_{t+1}$, the belief for the decision in period $t+1$.

\begin{center}
\textbf{Decide} $a_t$
      $\;\to\;$ \textbf{Observe} $x_t$
      $\;\to\;$ \textbf{Correct $\bvec{b}_t$}
      $\;\to\;$ \textbf{Predict $\bvec{b}_{t+1}$}
      $\;\to\;$ (period $t+1$)
\end{center}

In the \emph{correct} step, after observing demand $x_t$ at calendar time $(d_t, h_t)$, update via Bayes' rule:
\begin{equation}
    b_t^+(j) = \frac{b_t(j) \cdot f_j(x_t;\, d_t, h_t)}
                    {\displaystyle\sum_{k=1}^{K} b_t(k) \cdot f_k(x_t;\, d_t, h_t)},
    \quad j = 1, \ldots, K,
    \label{eq:correct}
\end{equation}
where $f_j(\cdot;\, d_t, h_t)$ is the emission for regime $j$ at the current calendar time. The posterior $\bvec{b}_t^+$ reflects the information from the observed demand.

The \emph{predict} step projects the posterior forward using the transition matrix to obtain the prior for period $t+1$, $\bvec{b}_{t+1} = \bvec{T}^\top \bvec{b}_t^+$. Combining both steps:
\begin{equation}
    \bvec{b}_{t+1} = \bvec{T}^\top \cdot \mathrm{Correct}(\bvec{b}_t,\, x_t;\, d_t, h_t),
    \label{eq:belief_combined}
\end{equation}
where $\mathrm{Correct}(\bvec{b}_t, x_t;\, d_t, h_t)$ denotes the posterior $\bvec{b}_t^+$ from Equation \eqref{eq:correct}.

The correct step reweights each regime by how plausibly it produced the observed demand, so a low $x_t$ relative to what was expected favors low-demand regimes. The predict step then advances the posterior one period through $\bvec{T}$, so regimes with a large diagonal persist while large off-diagonals spread mass. One modeling assumption here is that the transition matrix $\bvec{T}$ does not depend on the action $a_t$. Demand regimes evolve according to their own dynamics regardless of how many drivers are deployed, thus the action affects the reward and the backlog but not the latent state. This separability allows the belief to be updated from observables alone and is standard in POMDP models of demand uncertainty \citep{treharne2002adaptive}.

\subsection{Value Function Formulation}
\label{sec:bellman}

Since the observation model is store-specific, the value function and optimal policy are computed per store. We adopt a finite-horizon formulation in which each operating day constitutes an independent episode, with beliefs reset to the stationary distribution $\bvec{\pi}$ at the start of each day. Within-day persistence is strong (adjacent hours are highly correlated), whereas overnight transitions are weak, because the shocks driving intra-day regime dynamics (local events, weather, promotions) are largely resolved before the next operating day. The finite-horizon structure avoids a discount factor and aligns the planning cycle with the natural daily shift structure of dark-store operations.

Each operating day comprises $T$ periods (the number of operating hours), so the value function carries a time index. Since the Bellman equation depends on the day of week $d$ through the calendar baselines in the emission model, a separate backward induction is solved for each $d = 1, \ldots, 7$, yielding seven time-indexed policies per store. We make this dependence explicit by writing $V_t^d$. At period $t$ with calendar context $(d, h_t)$, the value function $V_t^d(s_t, \bvec{b}_t)$ satisfies
\begin{equation}
    V_t^d(s_t, \bvec{b}_t) = \max_{a_t} \left\{
        \sum_{k=1}^{K} b_t(k)
        \sum_{x=0}^{x_{\max}}
        \Bigl[ r(x, s_t, a_t) + V_{t+1}^d\!\bigl(s_{t+1}(x),\; \bvec{b}_{t+1}(x)\bigr) \Bigr]
        f_k(x;\, d, h_t)
    \right\},
    \label{eq:bellman}
\end{equation}
with terminal condition $V_{T+1}^d(s, \bvec{b}) = -c_{\mathrm{lost}} \cdot s$ (penalizing leftover backlog at close). Here $r(x, s_t, a_t)$ is the reward from Equation \eqref{eq:reward_spillover}. $s_{t+1}(x) = \max\{x + s_t - v \, a_t,\; 0\}$ is the next-period backlog (Equation \ref{eq:backlog_transition}), $f_k(x;\, d, h_t)$ is the emission probability mass at the current calendar time, $x_{\max}$ is a truncation point set high in the upper tail of the highest-regime emission, so that the omitted probability mass is negligible, and $\bvec{b}_{t+1}(x)$ is the belief at the start of period $t+1$, obtained by correcting $\bvec{b}_t$ with observation $x$ and predicting forward (Equation \ref{eq:belief_combined}).

Although Equation \eqref{eq:bellman} is written as a discrete sum over $x \in \{0, 1, \ldots, x_{\max}\}$, the expectation is evaluated via Monte Carlo sampling. For each regime $k$, demand realizations are drawn from the continuous emission $f_k(\cdot;\, d, h_t)$ and rounded to the nearest integer.

The belief $\bvec{b}_t$ enters Equation \eqref{eq:bellman} linearly since, for a fixed action and a fixed selection of continuation $\alpha$-vector at each observation, the right-hand side is affine in $\bvec{b}_t$, because the normalizing constant of the belief update cancels against the outer observation likelihood. This is the piecewise-linear-convex (PWLC) structure of POMDP value functions \citep{sondik1978optimal}, in which $V_t^d$ is the pointwise maximum over a finite set of $\alpha$-vectors (hyperplanes over the belief simplex), convex and piecewise linear in $\bvec{b}_t$. The PWLC structure guarantees that the value function is in principle exactly representable over the belief simplex, but in practice, the number of $\alpha$-vectors grows exponentially with the horizon. Our computational approach instead discretizes the belief simplex directly (Section \ref{sec:algorithm}), trading exact representation for tractability while preserving the property that the optimal policy is computed by backward induction on a finite grid.

\subsection{Structural Properties of the POMDP}
\label{sec:structural}

With the recursion in place, we now ask how its optimal policy behaves, before turning to how it is computed. Two questions decide whether that policy matches operational intuition. Does staffing rise as backlog accumulates rather than fall, and does it rise gradually rather than swing from hour to hour? A policy that staffed down as the queue grew would mean the backlog and terminal penalties are miscalibrated, and one that moved the roster sharply on small changes in backlog would be hard to run in practice.

To answer these questions, we characterize the value function and optimal policy. The results below use four standing conditions stated above as part of the model: (A1) the net per-order margin $q \coloneqq p - c_p > 0$; (A2) the backlog penalty dominates the margin, $c_b > q$; (A3) the regime transition matrix is action-independent; and (A4) the action upper bound $a_{\max}$ is non-binding.

The inner expression of \eqref{eq:bellman} defines the \emph{action-value function} (or $Q$-function), the expected total reward from period $t$ onward at state $(s_t, \bvec{b}_t)$ under action $a_t$ followed by optimal continuation:
\begin{equation}
    Q_t^d(s_t, \bvec{b}_t, a_t) \;:=\; \sum_{k=1}^{K} b_t(k)
    \sum_{x=0}^{x_{\max}}
    \Bigl[\, r(x, s_t, a_t) + V_{t+1}^d\!\bigl(s_{t+1}(x),\; \bvec{b}_{t+1}(x)\bigr) \Bigr]
    f_k(x;\, d, h_t),
    \label{eq:qfunction}
\end{equation}
so that $V_t^d(s_t, \bvec{b}_t) = \max_{a_t} Q_t^d(s_t, \bvec{b}_t, a_t)$. The structural results below are most cleanly stated and proved on $Q_t$, where action comparisons (cross-differences in $(s, a)$) appear directly without unpacking a max.

Throughout, $a^*_t(s,\mathbf{b}) := \min\{a \in \arg\max_{a' \in \mathcal{A}} Q_t(s,\mathbf{b},a')\}$ denotes the smallest-index optimal action (smallest-argmax convention), with $Q_t$ as defined in \eqref{eq:qfunction}.

\begin{restatable}[Backlog Monotonicity]{theorem}{thmBacklog}
\label{prop:backlog_monotone}
For every period $t\in\{1,\ldots,T\}$ and every belief $\mathbf{b}\in\Delta^{K-1}$:
\begin{enumerate}[label=(\roman*)]
    \item The map $s\mapsto V_t(s,\mathbf{b})$ is strictly decreasing.
    \item The smallest optimal allocation $a^*_t(s,\mathbf{b})$ is non-decreasing in $s$.
    \item The policy has bounded rate of increase: $a^*_t(s,\mathbf{b}) \leq a^*_t(s+v,\mathbf{b}) \leq a^*_t(s,\mathbf{b}) + 1$.
\end{enumerate}
\end{restatable}

Proof in \ref{app:proofs}.

Theorem \ref{prop:backlog_monotone} captures the core operational regularity of the model. A unit of backlog erodes value at a rate strictly exceeding the margin from clearing it, so backlog is unambiguously costly and the optimal staffing level rises smoothly in backlog, by at most one additional driver per $v$ carried-over orders. The proof utilizes a translation identity that exchanges $v$ extra backlog units for one extra driver at a fixed net cost, packaging the additional revenue, backlog penalty, and driver wage into a single constant. Topkis's theorem applied to the cross-difference of the $Q$-function then gives the result. For an operator this bounded-increase property is a stability guarantee. The roster never jumps by more than one driver for each $v$ orders of accumulated backlog, so the policy responds to congestion gradually rather than swinging the shift size from hour to hour.

\subsection{Solution Algorithm}
\label{sec:algorithm}
%%=========================================================================

Each operating day is a finite-horizon problem with $T$ operating hours. Backward induction computes an optimal policy $\pi(s_t, \bvec{b}_t, t, d)$ in a single backward pass per day-of-week $d$, with no convergence loop and no discount factor. At each stage $t = T, \ldots, 1$, the Q-function $Q_t(s, \bvec{b}, a)$ is evaluated via Monte Carlo. For each regime $k$, demand realizations are sampled from $f_k(\cdot;\, d, h_t)$, and the continuation value is obtained by barycentric interpolation on a regular grid over the belief simplex $\Delta^{K-1}$ with step $\Delta b = 0.05$ ($|\mathcal{B}| = \binom{K-1+1/\Delta b}{K-1}$ grid points). Full pseudocode is provided in  \ref{app:algorithms}. The total cost is $O(7 \cdot T \cdot |\mathcal{S}| \cdot |\mathcal{B}| \cdot |\mathcal{A}| \cdot K \cdot N_{\text{mc}})$, obtained from seven day-of-week policies, each requiring a single backward pass of $T \approx 16$ stages. With $K = 3$ regimes, the belief grid contains $|\mathcal{B}| = 231$ points; combined with $s_{\max}$ backlog levels the state space remains tractable.

At execution time, the belief resets to the stationary distribution $\bvec{\pi}$ of $\bvec{T}$ at the start of each operating day and is updated period-by-period via the correct--predict cycle (Equation \ref{eq:belief_combined}). The policy lookup $a_t = \pi(s_t, \bvec{b}_t, t, d)$ and belief update are both $O(K)$ operations per period.

%%=========================================================================
\section{Data and Parameter Estimation}
\label{sec:demand_estimation}
%%=========================================================================

We validate the framework using operational data from SuperGlovo, the grocery delivery vertical of Glovo, a Spain-based technology platform founded in 2015. SuperGlovo operates micro-fulfillment centers, commonly called dark stores, that carry a curated assortment of grocery products and dispatch orders through contracted full-time driver fleets, with deliveries typically completed within 15 to 30 minutes. The dataset comprises 27 micro-fulfillment centers across 13 cities in Spain, covering 2021--2022. Two cities host eight stores each, and the remaining eleven cities contain a single store. The time series reflects Glovo's phased expansion, beginning with the two multi-store cities and gradually extending to additional locations, resulting in an unbalanced panel. Each record is an hourly time-stamped observation of demand (number of orders), but backlog is not directly recorded. Operating hours vary by store and day of week. Demand is recorded only during active business hours, so supplementary data on store-specific schedules provided by Glovo were incorporated. Records with incomplete schedule metadata were excluded from the analysis.

Table \ref{table:descriptive} in the appendix reports descriptive statistics of hourly demand by city, including mean, standard deviation, maximum, time-series length, KPSS statistic, and approximate entropy \citep{olbrys2022approximate}. Demand levels and variability differ across cities, reflecting differences in urban density and store concentration. The two multi-store cities show higher mean demand, motivating store-specific model parameters. Approximate entropy values indicate that demand irregularity is pervasive across all cities, and the autocorrelation of calendar-adjusted residuals shows persistent structure that calendar features alone cannot explain. Together, these patterns motivate the latent regime structure estimated below. Figure \ref{figure:seas} illustrates the hourly and weekly seasonality for a representative city.

\begin{figure}[htbp]
    \centering
    % First Subfigure: Allocate 49% of the text width
    \begin{subfigure}[b]{0.45\textwidth}
        \centering
        % Note: The width here is relative to the subfigure environment, so we use \textwidth or \linewidth to fill it 100%.
        \includegraphics[width=\textwidth]{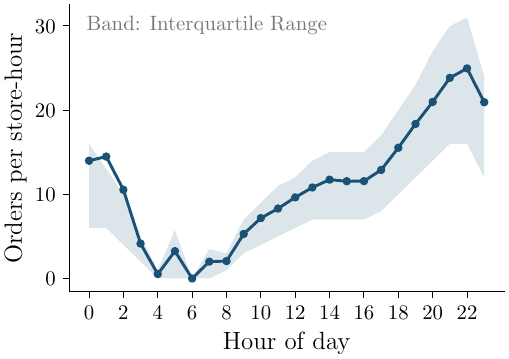}
        \caption{Hourly Demand for City B}
        \label{figure:hourly_seas}
    \end{subfigure}
    \hfill % Pushes the subfigures to the edges, creating a natural gap in the remaining 2% of space
    % Second Subfigure: Allocate the other 49%
    \begin{subfigure}[b]{0.45\textwidth}
        \centering
        \includegraphics[width=\textwidth]{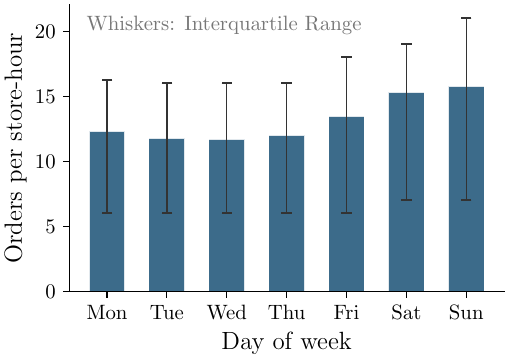}
        \caption{Weekday Demand for City B}
        \label{figure:weekday_seas}
    \end{subfigure}
    
    \caption{Hourly and weekly demand distribution for a representative city. Demand patterns are consistent across all stores; City B is shown for illustration.}
    \label{figure:seas}
\end{figure}

Demand exhibits strong calendar seasonality. Order volumes peak in the early evening and late night, remain steady from late morning through late afternoon, and drop substantially between 1\,AM and 8\,AM when most stores are closed. Weekly patterns show higher demand on weekends than weekdays. Conditional on the calendar period, hourly demand is well described by a log-normal distribution. A Kolmogorov-Smirnov test cannot reject this fit in over 75\% of store-weekday-hour combinations, supporting its use as the emission distribution. Beyond these predictable calendar effects, demand displays within-day persistence. This persistence, visible in the autocorrelation of calendar-adjusted residuals, is what the latent regime structure is designed to capture.

An institutional feature of the data is Spain's \textit{Ley Rider} (Real Decreto-ley 9/2021), enacted in August 2021 and which established a presumption of employment,
reclassifying platform delivery couriers as employees rather than independent contractors
\citep{VieiraMendonca2025}. This regulatory shift, part of a broader European trend toward
platform-worker protections \citep{europarlGigEconomy}, prompted delivery platforms to
restructure their Spanish workforces around this time. Deliveroo exited the market
(over 3{,}000 riders), while Glovo overhauled its app and moved only about 2{,}000 of its
roughly 12{,}000 couriers onto employee contracts, scaling back its active self-employed
roster \citep{Aranguiz2021, VieiraMendonca2025}.

\subsection{Estimation Pipeline}
\label{sec:estimation_overview}

The POMDP formulation in Section \ref{sec:pomdp} requires three inputs: the emission model $f_k(x;\, d, h)$, the transition matrix $\bvec{T}$, and the initial belief $\bvec{b}_1$. The estimation pipeline is applied independently to each store, with one exception. The emission model regime parameters are estimated by pooling demand shocks across all stores. The core challenge is to learn a compact, probabilistic representation of demand regimes that (i) reflects genuine demand variation rather than calendar structure, (ii) is estimated from pooled data across all stores so parameters are well-identified, and (iii) produces per-store, per-slot emission likelihoods that depend on the current calendar context.

We have 93,417 training observations (open stores only) pooled across 27 stores. Not all (store, day-of-week, hour) combinations are observed: 3,126 of the theoretical $27 \times 7 \times 24 = 4{,}536$ tuples appear in the data (stores have limited operating hours), yielding approximately 30 observations per tuple. For each observed (day-of-week, hour) pair $(d, h)$, define the \emph{calendar baseline} as the per-pair sample mean:
\begin{equation}
    \mu_{d,h} = \frac{1}{n_{d,h}} \sum_{\substack{t:\,(d_t,\,h_t)\\=(d,\,h)}} x_t,
    \label{eq:lambda}
\end{equation}
where $n_{d,h}$ is the number of training observations for that pair and $x_t$ is realized demand. This sample mean is the minimum-variance unbiased estimator of $E[x_t \mid d, h]$ under any distribution, requires no modeling assumptions, and serves as the calendar baseline in all subsequent steps.

The calendar baseline $\mu_{d,h}$ captures the \emph{average} demand level for each operating slot. Deviations from this baseline (demand running unexpectedly high or low) are what the latent state should track. We define the \emph{multiplicative demand shock} for observation $t$ as
\begin{equation}
    z_t = \frac{x_t + 1}{\mu_{d,h} + 1}.
    \label{eq:shock}
\end{equation}
A value $z_t \approx 1$ means demand matched the calendar expectation, $z_t > 1$ means demand exceeded it, and $z_t < 1$ means demand fell short. The $+1$ shift is a continuity correction that avoids division by zero when $\mu_{d,h} = 0$ and stabilizes the ratio for small baselines. Because shocks are calendar-normalized, they are comparable across stores of different sizes: a shock of 2.0 means ``twice the expected demand'' regardless of store scale. Because a multiplicative shock is naturally modeled on the log scale, we work with the log-shock $y_t = \log z_t$, which preserves the multiplier interpretation and makes a Gaussian mixture the natural emission model.

\subsubsection*{Regime Clustering and Emission Model}
\label{sec:gmm}

The shocks $z_t$ defined above strip away predictable calendar variation, isolating the residual demand fluctuation that the latent state should capture. Clustering these shocks, rather than raw demand or calendar-level summary statistics, ensures that the resulting regimes represent genuine deviations from expected demand, rather than differences in store size or time-of-day effects. A Gaussian Mixture Model (GMM) is a natural choice for this task. It is a probabilistic generative model that directly provides the emission likelihoods required for belief updates in the POMDP, produces soft (probabilistic) regime assignments that avoid the boundary distortion of hard clustering, and requires no separate distribution-fitting step to recover per-regime densities. We pool all 93,417 training log-shocks $y_t = \log z_t$ across all stores and fit a 1D GMM with $K$ components:
\begin{equation}
    p(y) = \sum_{k=1}^{K} \pi_k^{\text{GMM}} \cdot \mathcal{N}(y;\, \theta_k,\, \sigma_k^2).
    \label{eq:gmm}
\end{equation}
After fitting, sort components by mean so that $\theta_1 < \theta_2 < \cdots < \theta_K$ (component 1 is the lowest-demand regime, component $K$ the highest). For every observation $t$, compute the responsibility vector:
\begin{equation}
    \gamma_t(k) = \frac{\pi_k^{\text{GMM}} \cdot \mathcal{N}(y_t;\,\theta_k,\,\sigma_k^2)}
                       {\displaystyle\sum_{j=1}^{K} \pi_j^{\text{GMM}} \cdot \mathcal{N}(y_t;\,\theta_j,\,\sigma_j^2)},
    \quad k = 1,\ldots,K.
    \label{eq:responsibility}
\end{equation}
Each $\gamma_t(k) \in (0,1)$ is the posterior probability that observation $t$ was generated by regime $k$.

The GMM fit yields two objects. The first is the set of shared regime parameters $(\theta_k, \sigma_k)_{k=1}^K$, used in the emission model below. The second is a $(93,417 \times K)$ responsibility matrix $\bm{\Gamma}$ with rows $\bm{\gamma}_t$, to initialize the per-store transition matrix estiamtion (Section \ref{sec:transition}). The mixing weights $\pi_k^{\text{GMM}}$ serve as a consistency check against the stationary distribution of each store's transition matrix.

No separate fitting step is needed for the emission model. The log-shock definition $y_t = \log\left((x_t + 1)/(\mu_{d,h}+1)\right)$ and the GMM model $y_t \mid \text{regime }k \sim \mathcal{N}(\theta_k, \sigma_k^2)$ together imply:
\begin{equation}
    x_t + 1 \mid \text{regime } k,\; d,\; h \;\sim\;
    \mathrm{LogNormal}\!\left(\theta_k + \log(\mu_{d,h}+1),\;\; \sigma_k^2\right),
    \label{eq:x_given_k}
\end{equation}
with log-mean $\theta_k + \log(\mu_{d,h}+1)$ and log-standard-deviation $\sigma_k$. The time-specific emission is therefore:
\begin{equation}
    f_k(x;\, d, h) =
    \mathrm{LogNormal}\!\left(x+1;\;\;
    \underbrace{\sigma_k}_{\text{shape}},\;\;
    \underbrace{e^{\theta_k}(\mu_{d,h}+1)}_{\text{scale}}\right).
    \label{eq:emission}
\end{equation}
The calendar baseline $\mu_{d,h}$ sets the baseline demand level, $e^{\theta_k}$ is the shared regime multiplier, and $\sigma_k$ is the shared log-scale spread, both from the GMM. For large $\mu_{d,h}$ the median demand scales as $e^{\theta_k}\mu_{d,h}$, and the coefficient of variation $\sqrt{e^{\sigma_k^2}-1}$ depends only on the regime, not the store.

The estimation pipeline cleanly separates shared and per-store quantities. The regime parameters $(\theta_k, \sigma_k)_{k=1}^K$ and mixing weights $\pi_k^{\text{GMM}}$ are shared across all stores and estimated once from all pooled shocks. The calendar baselines $\mu_{d,h}$, transition matrix $\bvec{T}$ (Section \ref{sec:transition}), emission $f_k(x;\, d, h)$, and policy $\pi$ are per-store. This separation is natural because the regime dynamics (how demand evolves across hours) are structural and similar across stores, while the calendar baseline captures the level. High demand may mean 50 orders in one store and 10 in another.

The number of latent regimes is selected using the Bayesian Information Criterion (BIC), which penalizes model complexity to guard against overfitting. The BIC drops sharply from $K=1$ to $K=2$ ($-17{,}322$), improves meaningfully at $K=3$ ($-1{,}345$), and flattens thereafter. The gain from $K=3$ to $K=4$ is only $-581$, less than a half of the preceding drop, and a fifth component adds almost nothing (-54). We therefore select $K=3$.

The three regimes, sorted by mean, correspond to multiplicative shocks to the calendar mean demand $\mu_{d,h}$. A shock of 1.0 means current demand mean matched the calendar mean exactly, while values below or above 1.0 indicate demand mean falling short of or exceeding the calendar mean for that particular store. Regime 0 ($e^{\theta_0} \approx 0.58$) captures hours where demand runs roughly 42\% below the calendar mean. Regime 1 ($e^{\theta_1} \approx 0.78$) captures demand running about 22\% below the mean, and Regime 2 ($e^{\theta_2} \approx 1.18$) captures demand approximately 18\% above. The mixing weights indicate that most operating hours fall in Regimes 1 and 2 ($\pi_1 = 0.43$, $\pi_2 = 0.45$), with Regime 0 accounting for the remaining 12\% of hours.

Figure \ref{figure:GMM_mixture_plot} displays the histogram of observed log-shocks with the fitted three-component mixture density overlaid. The black curve shows the overall mixture fit, which closely tracks the empirical distribution across the full support. The individual component densities have clearly separated means, though their supports overlap substantially. This overlap is what limits how sharply a single hourly observation can identify the active regime.

\begin{figure}[htbp]
    \centering
    \includegraphics[width=0.8\textwidth]{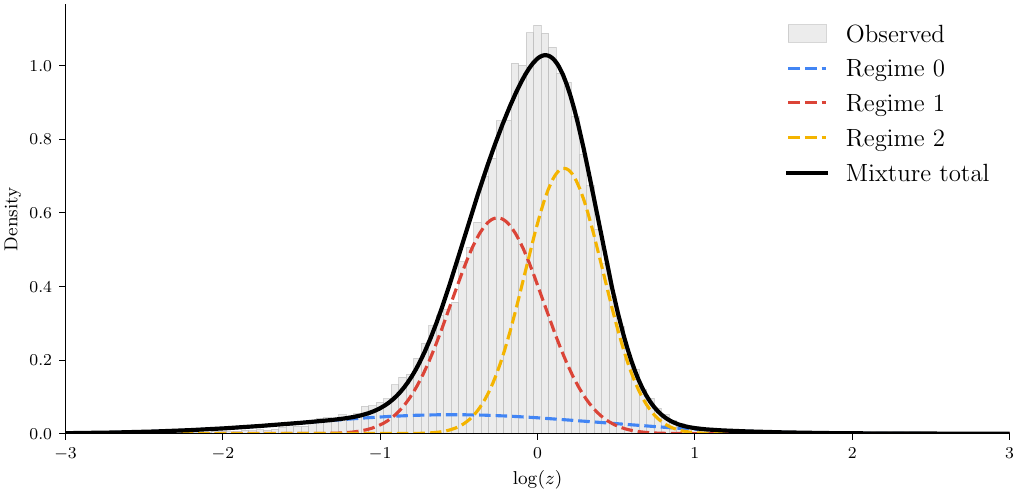}
    \caption{Histogram of observed log-shocks with the fitted
    three-component Gaussian mixture density.}
    \label{figure:GMM_mixture_plot}
\end{figure}

\subsubsection*{Transition Dynamics}
\label{sec:transition}

We estimate one transition matrix $\bvec{T}$ per store by maximum likelihood, holding the pooled GMM emission components $(\theta_k, \sigma_k)$ from Section \ref{sec:gmm} fixed. The only quantity that evolves stochastically across operating hours is the active demand regime, the latent state whose multiplier $e^{\theta_k}$ scales the calendar baseline up or down. The matrix $\bvec{T}$ therefore governs the regime and not the demand level directly, and it captures how the regime persists or shifts from one operating hour to the next within a day. Including overnight pairs would inject spurious structure into $\bvec{T}$, so we restrict to consecutive operating-hour pairs within the same day. From an average of 3,460 training observations per store, removing the last observation of each operating day (which has no valid ``next hour'' within that day) leaves an average of 3,120 valid within-day consecutive pairs per store. For $K = 3$ (6 free parameters per matrix after row-normalization), this is well-powered.

We use the Baum--Welch algorithm to estimate $\bvec{T}$, a standard expectation-maximization procedure for hidden Markov models \citep{baum1970maximization, rabiner1989tutorial}. It holds the emission components fixed at the pooled GMM estimates, tying the initial distribution of each day to the stationary distribution of $\bvec{T}$, which matches the daily belief reset in the POMDP. %Appendix \ref{app:estimation_bias} further supports this choice versus a simpler classify-then count rule.

Figure \ref{figure:transition_matrix_store} displays the estimated transition matrix for store B6. The dominant pattern is strong diagonal persistence. Each regime renews itself with probability 0.64 to 0.88 from one hour to the next, and the second eigenvalue of the matrix is 0.77, so the information carried by an hourly observation decays with a half-life of roughly 2.6 hours. Regime 2 is the most absorbing state and the stationary distribution places 51\% of its weight there, with 45\% on Regime~1 and about 5\% on the low-demand Regime~0. Because regimes are persistent but not permanent, a single hour of unusually high demand is informative about the next several hours without being deterministic. Whether the POMDP can translate this predictability into meaningful staffing gains is an empirical question addressed in Section~\ref{sec:single_store_benchmarks}. In the POMDP computation, each store uses its own estimated transition matrix $\bvec{T}$ rather than a pooled average, allowing the policy to adapt to store-specific regime dynamics.
The starting belief $\bvec{b}_1$ is set to the stationary distribution $\bvec{\pi}$ of $\bvec{T}$, satisfying $\bvec{\pi}^\top \bvec{T} = \bvec{\pi}^\top$. This is the principled uninformative prior for the first operating hour of a session.

\begin{figure}[htbp]
   \centering
   \includegraphics[width=0.4\textwidth]{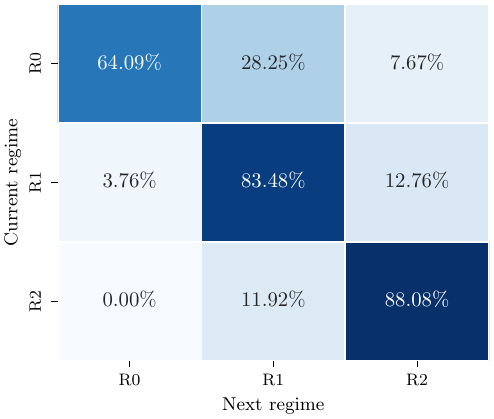}
   \caption{Estimated transition matrix for store B6.}
   \label{figure:transition_matrix_store}
\end{figure}

Two checks confirm that this persistence is a property of the data rather than an artifact of the estimator. The Baum--Welch matrices improve one-step-ahead predictive likelihood on the held-out test period at 25 of the 27 stores, and the log-shock autocorrelation they imply closely matches the autocorrelation observed in the data. The estimated stationary distribution is also broadly consistent with the GMM mixing weights for the middle and high regimes, placing less mass on the rare low-demand regime once persistence is accounted for. %By contrast, the simpler classify-then-count estimator recovers $|\lambda_2| \leq 0.15$ on synthetic data regardless of the true persistence, so correct estimation is what makes the information value of Section~\ref{sec:single_store_benchmarks} measurable at all (Appendix~\ref{app:estimation_bias}).

% \begin{figure}[htbp]
%     \centering
%     \begin{minipage}[c]{0.45\textwidth}
%         \centering
%         \includegraphics[width=\textwidth]{figures/fig_transition_matrix.pdf}
%         \caption{Estimated transition matrix for store~A.}
%         \label{figure:transition_matrix_store}
%     \end{minipage}
%     \hfill 
%     \begin{minipage}[c]{0.50\textwidth}
%         Two checks confirm that this persistence is a property of the data rather than an artifact of the estimator. The Baum--Welch matrices improve one-step-ahead predictive likelihood on the held-out test period at 25 of the 27 stores, and the log-shock autocorrelation they imply closely matches the autocorrelation observed in the data. The estimated stationary distribution is also broadly consistent with the GMM mixing weights for the middle and high regimes, placing less mass on the rare low-demand regime once persistence is accounted for.
%     \end{minipage}
% \end{figure}

%The starting belief $\bvec{b}_1$ is set to the stationary distribution $\bvec{\pi}$ of $\bvec{T}$, satisfying $\bvec{\pi}^\top \bvec{T} = \bvec{\pi}^\top$. This is the principled uninformative prior for the first operating hour of a session.

%%=========================================================================
\section{Computational Results}
\label{sec:single_store_results}
%%=========================================================================

\subsection{Experimental Setup}
\label{sec:experimental_setup}

We calibrate the model parameters to reflect the operating economics of a quick grocery delivery service. Exact cost and revenue figures are confidential, so the values used are representative approximations validated by Glovo's operations team. Table~\ref{tab:base_params} collects all base-case values. Revenue per fulfilled order is set to $p = 20$, the hourly driver wage to $c_a = 15$, and the picking cost to $c_p = 5$ per fulfilled order. Together these imply a positive margin per driver-hour when a driver is fully utilized: $vp - vc_p - c_a = 2(20) - 2(5) - 15 = 15 > 0$, satisfying the viability condition stated in Section~\ref{sec:pomdp}. Driver capacity is set to $v = 2$ orders per driver per hour, reflecting the operational reality that each delivery involves travel time, handoff, and return, limiting throughput. The backlog penalty $c_b = 18$ per carried-over order per period satisfies $c_b > p - c_p = 15$, ensuring that carrying backlog is never preferable to fulfilling demand on time. Without this condition, the model would perversely incentivize accumulating backlog.

\begin{table}[t]
\centering\small
\caption{Base-case parameter values for single-store experiments.}
\label{tab:base_params}
\begin{tabular}{@{}llrl@{}}
\toprule
\textbf{Symbol} & \textbf{Description} & \textbf{Value} & \textbf{Unit} \\
\midrule
\multicolumn{4}{@{}l}{\emph{Economic parameters}} \\[2pt]
$p$   & Revenue per fulfilled order        & 20  & \$/order \\
$c_a$ & Driver wage                        & 15  & \$/driver-hour \\
$c_p$ & Picking cost per fulfilled order   & 5   & \$/order \\
$c_b$ & Backlog penalty                    & 18  & \$/order/period \\
$c_{\text{lost}}$ & Terminal lost-sale penalty  & 25  & \$/order \\
$v$   & Driver capacity                    & 2   & orders/driver/hour \\
\midrule
\multicolumn{4}{@{}l}{\emph{Demand model}} \\[2pt]
$K$   & Number of regimes                  & 3   & --- \\
\midrule
\multicolumn{4}{@{}l}{\emph{Discretization and computation}} \\[2pt]
$a_{\max}$ & Maximum drivers             & 50  & drivers \\
$s_{\max}$ & Maximum backlog             & 30  & orders \\
$\Delta b$ & Belief-grid step            & 0.05 & --- \\
$|\mathcal{B}|$ & Belief-grid points     & 231 & --- \\
$N_{\text{mc}}$ & Monte Carlo samples per cell & 1{,}000/300 & --- \\
\bottomrule
\end{tabular}
\end{table}

The action space ranges from $a = 0$ (no drivers) to $a_{\max} = 50$. At full staffing the store can fulfill $v \cdot a_{\max} = 100$ orders per hour. Across the panel, hourly demand exceeds this capacity ceiling in approximately 0.09\% of operating hours, so the constraint is active only during rare peak episodes. Backlog takes integer values $s \in \{0, 1, \ldots, 30\}$. The upper bound $s_{\max} = 30$ acts as a truncation: any backlog exceeding 30 orders is capped at the boundary, with the terminal penalty $c_{\text{lost}}$ applying to the full unserved quantity at end of day. Both $a_{\max}$ and $s_{\max}$ are set for computational tractability. In the validation (Section~\ref{sec:single_store_benchmarks}), the POMDP policy stays below $a_{\max}$, peaking near 28 drivers in the evening, and reaches the backlog cap $s_{\max}$ only in rare peak hours. Because all benchmark policies are evaluated under the same state-space truncation and on the same demand realizations, the relative performance comparison remains valid even where the bounds are reached.

The belief simplex $\Delta^{K-1}$ is discretized with step $\Delta b = 0.05$, yielding $|\mathcal{B}| = \binom{22}{2} = 231$ grid points. Because belief updates produce continuous vectors that do not coincide with grid points, value-function lookups use barycentric interpolation on the simplex grid, as described in Section~\ref{sec:algorithm}. The same interpolation is applied at execution time to the stored Q-function.

%%=========================================================================
\subsubsection{Evaluation Protocol}
\label{sec:Evaluation_protocol}
%%=========================================================================

To prevent information leakage, all model components are estimated on a training set and validated on a held-out test set using a temporal split. Dates are sorted chronologically, with the first 80\% forming the training set and the remaining 20\% forming the test set. A temporal rather than random split is the standard protocol for time-series evaluation, ensuring that the model is never trained on future data. Because the panel is unbalanced (stores enter at different times), the calendar cutoff date is common across stores, so some stores contribute fewer train observations than others. On the
training data we estimate the pooled GMM regime parameters $(\theta_k, \sigma_k)$ (Equation~\ref{eq:gmm}) once across all stores and, independently for each store, the calendar baselines $\mu_{d,h}$ (Equation~\ref{eq:lambda}) and the transition matrix $\bvec{T}$ (Section~\ref{sec:transition}). Backward induction (Algorithm~\ref{alg:backward_induction}) is then solved for each of the seven days of the week using the training-set parameters, producing day- and time-indexed policy and Q-function tables.

We select store B6, one of the eight stores in City B, as the representative store for the single-store analysis. It operates all seven days of the week, giving a complete schedule with no missing day-of-week combinations, and its demand profile exercises the policy’s full staffing range, from a handful of drivers at opening to nearly thirty at the Sunday evening peak.

The computational cost of backward induction is incurred once, offline, during the training phase. Once the policy and Q-function tables are computed and stored, the online phase reduces to a simple table lookup. At each period, the action is determined from the current state $(s, \bvec{b}, t)$ using the precomputed Q-function and barycentric interpolation, an $O(K)$ operation per period. This two-phase structure makes the policy practical for real-time deployment, even under the tight decision cadence of hourly staffing updates. Each cell of the Q-function (Equation~\ref{eq:qfunction}) is estimated with $N_{\text{mc}} = 300-1{,}000$ (cross-store sweep and focal store analysis) Monte Carlo draws per regime. Doubling $N_{\text{mc}}$ changes the optimal action in fewer than 4\% of state-belief-time cells, confirming adequate convergence.

Validation proceeds via forward evaluation (Algorithm~\ref{alg:online}) on the test-set demand realizations. At the start of each operating day the belief is reset to the stationary distribution $\bvec{\pi}$, backlog is reset to zero, and the precomputed policy prescribes driver counts hour by hour via the table lookup described above. Rewards are computed using the realized test-set demand (not simulated demand), so the comparison reflects out-of-sample performance under actual demand conditions. Each benchmark policy listed below is evaluated on the identical demand sequence with independent backlog dynamics, ensuring a controlled comparison in which only the staffing policy differs.

We compare three policies, each adding one modeling ingredient:

\begin{enumerate}

    \item \textbf{Myopic.} Each hour is treated independently with no backlog dynamics and no regime learning. The staffing decision maximizes single-period expected profit using the calendar baseline $\mu_{d,h}$ as the demand forecast. This serves as the comparison baseline. It is a calendar-only reference point and a conservative lower bound, since a practitioner who watches the current queue would staff up as backlog builds.

    \item \textbf{Static-Belief POMDP.} The full dynamic program is solved but the belief is fixed at the stationary prior $\bvec{\pi}$ and never updated. The policy accounts for backlog dynamics but ignores within-day regime learning. The gap between this and Myopic isolates the benefit of multi-period optimization accounting for backlog dynamics.

    \item \textbf{POMDP (Full Model).} The full dynamic program with Bayesian belief updating after each hourly observation. The gap between this and the Static-Belief policy isolates the benefit of regime learning and the value of information.

\end{enumerate}

%%=========================================================================
\subsection{Structural Properties of the Optimal Policy}
\label{sec:structural_results}
%%=========================================================================

The structural theorem in Section \ref{sec:structural} is proved for the exact finite-horizon Bellman recursion. The discretized solver, which uses Monte Carlo sampling of the emission and barycentric interpolation on the belief grid, need not preserve these properties exactly, so the figures in this subsection are an empirical check rather than a guarantee. All figures in this subsection are computed for store B6 on Sundays, the highest-demand day (mean hourly baseline
$\bar{\mu} = 11.0$).

\begin{figure}[htbp]
    \centering
    \begin{subfigure}[b]{0.48\textwidth}
        \centering
        \includegraphics[width=\textwidth]{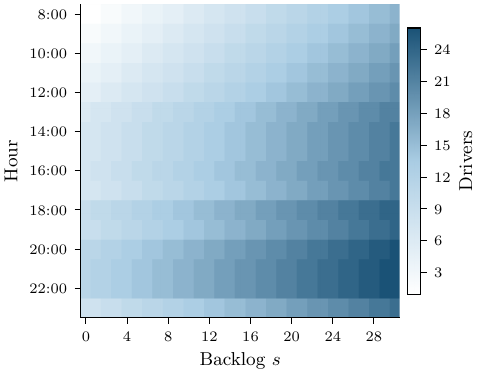}
        \caption{Optimal action $a^*(s, \bvec{b}_0, t)$}
        \label{fig:action_hour_backlog}
    \end{subfigure}
    \hfill
    \begin{subfigure}[b]{0.48\textwidth}
        \centering
        \includegraphics[width=\textwidth]{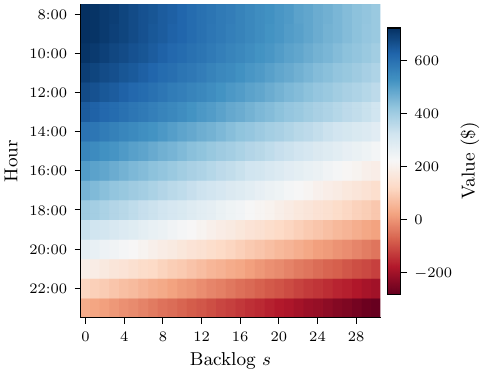}
        \caption{Value function $V(s, \bvec{b}_0, t)$}
        \label{fig:value_hour_backlog}
    \end{subfigure}
    \caption{Optimal policy and value function as a function of backlog
    and operating hour, with belief fixed at the stationary prior
    $\bvec{b}_0 = \bvec{\pi}$.}
    \label{fig:hour_backlog}
\end{figure}

Figure~\ref{fig:hour_backlog} displays the optimal policy and value function across backlog and operating hours, with belief fixed at the stationary prior. First, note the optimal number of drivers is non-decreasing in backlog at every hour, consistent with Theorem \ref{prop:backlog_monotone}. Second, the staffing level tracks the calendar demand profile, rising from a handful of drivers in the early morning through the midday peak and reaching its highest levels during the evening rush. The policy here has room to differentiate across the full range of demand levels, deploying more drivers during the evening peak than during the morning. Even at the highest backlog levels during the evening peak, the allocation tops out at 28 drivers, well below the action cap of 50. The policy also reduces staffing in the final hour as demand falls and the terminal horizon removes the need to protect against future backlog.

The value function (Figure~\ref{fig:value_hour_backlog}) confirms that backlog is unambiguously costly. Also consistent with Theorem \ref{prop:backlog_monotone}, we see that the expected remaining profit is strictly decreasing in $s$ at every hour, and values turn negative in the late evening once backlog is high, from $s \geq 26$ at 20:00 and widening toward closing. This end-of-day pressure is what drives the policy to staff aggressively during the evening, clearing the queue while there are still hours remaining to recover the cost.

\begin{figure}[htbp]
    \centering
    \begin{subfigure}[b]{0.48\textwidth}
        \centering
        \includegraphics[width=\textwidth]{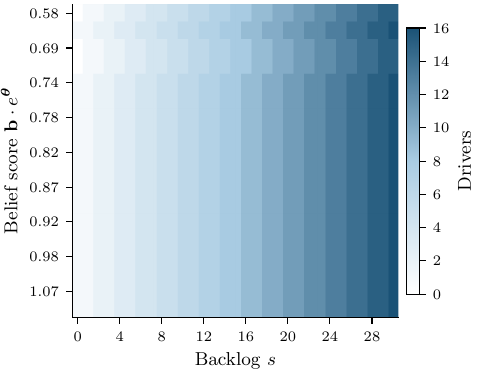}
        \caption{Optimal action $a^*(s, \bvec{b}, t_0)$}
        \label{fig:action_belief_backlog}
    \end{subfigure}
    \hfill
    \begin{subfigure}[b]{0.48\textwidth}
        \centering
        \includegraphics[width=\textwidth]{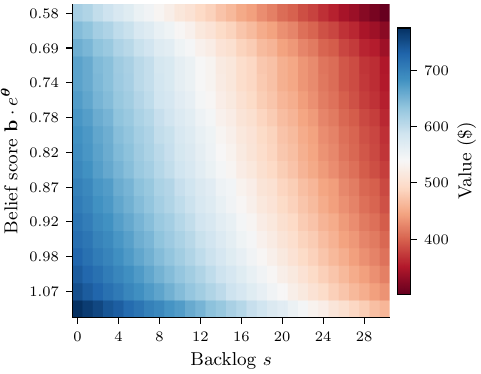}
        \caption{Value function $V(s, \bvec{b}, t_0)$}
        \label{fig:value_belief_backlog}
    \end{subfigure}
    \caption{Optimal policy and value function at the first operating hour ($t_0 = 8$:00) as a function of backlog $s$ and belief $\bvec{b}$. Rows are ordered from low-demand beliefs at the top to high-demand beliefs at the bottom. The ordering uses the scalar score $\bvec{b} \cdot e^{\bvec{\theta}} = \mathbb{E}[e^{\bvec{\theta}} \mid \bvec{b}]$, the posterior mean of the multiplier $e^{\theta_k}$, where each $e^{\theta_k}$ scales the calendar baseline $\mu_{d,h}$ to the conditional median demand of regime $k$ (Equation~\eqref{eq:emission}).}
    \label{fig:belief_backlog}
\end{figure}

Figure~\ref{fig:belief_backlog} fixes the hour at $t_0$ and varies the belief state. The policy is non-decreasing in backlog at every belief level, consistent with the pattern in Figure~\ref{fig:hour_backlog} and with Theorem \ref{prop:backlog_monotone}. Along the belief dimension, the action oscillates between adjacent values rather than increasing smoothly. The scalar score $\bvec{b} \cdot e^{\bvec{\theta}}$ records only the posterior mean demand multiplier, whereas the optimal staffing responds to the entire belief distribution, and in particular to how much probability mass sits on the high-demand regime and to the resulting variance of next-period demand. Two beliefs with the same mean score can carry different surge risk, so ordering states by a single scalar collapses beliefs that the policy treats differently, and the action need not move monotonically along that ordering. The size of the belief response also varies over the day. At the first hour it is narrow, shifting by at most one driver across the full belief spectrum at any fixed backlog, while by the evening peak the action spans up to five drivers across the simplex, once the filtered belief has had time to drift from the prior.

%%=========================================================================
\subsection{Value of Information Across Stores}
\label{sec:single_store_benchmarks}
%%=========================================================================

We evaluate the Myopic, Static-Belief, and Full POMDP policies across all 27 stores over the test period. Table~\ref{tab:cross_store} reports the aggregate results.

\begin{table}[htbp]
\centering\small
\caption{Cross-store policy comparison.}
\label{tab:cross_store}
\begin{tabular}{@{}lrrr@{}}
\toprule
& \textbf{Myopic} & \textbf{Static-Belief} & \textbf{Full POMDP} \\
\midrule
Total reward (\$)
  & $-3{,}212{,}070$ & $+1{,}730{,}501$ & $+1{,}918{,}981$ \\
Advantage over Myopic (\$)
  & --- & $+4{,}942{,}571$ & $+5{,}131{,}051$ \\
\bottomrule
\end{tabular}
\end{table}

The Static-Belief POMDP accumulates \$4.94M in additional reward over Myopic across the 27 stores. By accounting for backlog dynamics and the terminal penalty, this policy staffs more aggressively when the queue is building, preventing the compounding losses that a period-by-period rule ignores. Figure \ref{fig:cumulative_profit} traces cumulative total reward over the test period, giving each finding its own panel and scale. Panel \ref{fig:cumulative_profit_a} shows the value of multi-period optimization. Myopic staffing loses \$3.21M, since a calendar-based rule is unprofitable at 19 of the 27 stores once unfulfilled orders are allowed to compound as backlog, while the Static-Belief dynamic program earns \$1.73M on the same demand sequence. Panel \ref{fig:cumulative_profit_b} isolates the value of information. Real-time belief updating lifts the \$1.73M to \$1.92M, and the shaded wedge between the two trajectories widens steadily throughout the test period rather than being driven by a few outlier days.

The Full POMDP thus adds a further \$188K over the Static-Belief policy, a 10.9\% increase relative to Static-Belief. The gain is positive at 25 of the 27 stores and is concentrated where the estimated regimes are most persistent. The two negative entries are smaller than the Monte-Carlo and sampling error of the evaluation, roughly one thousand dollars per store, so their sign is not informative as both lie inside the shaded noise band in Figure~\ref{fig:voi_mechanism}

\begin{figure}[htbp]
    \centering
    \begin{subfigure}[b]{0.45\textwidth}
        \centering
        \includegraphics[width=\textwidth]{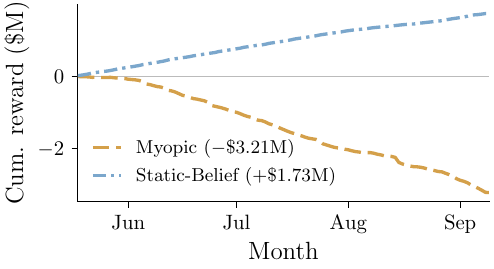}
        \caption{Multi-period optimization (Value of Backlog)}
        \label{fig:cumulative_profit_a}
    \end{subfigure}
    \hfill
    \begin{subfigure}[b]{0.45\textwidth}
        \centering
        \includegraphics[width=\textwidth]{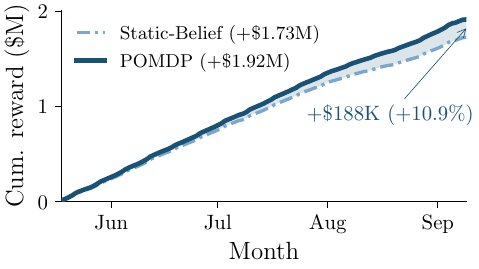}
        \caption{Regime Belief Learning (Value of Information)}
        \label{fig:cumulative_profit_b}
    \end{subfigure}
    \caption{Cumulative Reward Over Time}
    \label{fig:cumulative_profit}
\end{figure}

% \begin{figure}[htbp]
%     \centering
%     \includegraphics[width=0.95\textwidth]{figures/fig_cumulative_profit.pdf}
%     \caption{Cumulative total reward across all 27 stores.
%     (a)~Multi-period optimization: calendar-based Myopic staffing
%     loses \$3.21M as backlog compounds, while the Static-Belief
%     dynamic program earns \$1.73M. (b)~Value of information:
%     real-time belief updating (Full POMDP) adds a further \$188K
%     over Static-Belief, or 10.9\% of realized profit (shaded
%     wedge).}
%     \label{fig:cumulative_profit}
% \end{figure}

Figure~\ref{fig:voi_mechanism} plots each store's belief-updating gain against the persistence of its estimated regimes, measured by the second-largest eigenvalue modulus $|\lambda_2(\bvec{T})|$ (Spearman $\rho = 0.58, p \approx 0.002, n = 27$). This metric reinforces the idea that the stickier the regimes, the more learning is worth. Intuitively, $|\lambda_2|$ measures how long what a store just learned about its demand state stays true. Stores whose elevated demand comes from persistent conditions such as weather, local events, or supply disruptions reward real-time tracking, while stores whose hourly variation is mostly noise around the calendar gain nothing from it. The stickiest store's gain reaches \$41K, or 49\% of that store's POMDP reward. Because persistence is observable at estimation time, it identifies in advance the stores where real-time learning is worth its implementation cost.

At the four circled stores the Static-Belief policy earns less than the one-period Myopic newsvendor, by up to \$27K. A frozen belief implicitly assumes demand reverts toward its long-run average each hour, but at sticky stores where the regime persists, the lookahead optimizes against the wrong forecast (the belief goes back to stationary every period) and its backlog-aware adjustments amplify the error rather than correct it. The effect strengthens with persistence (Spearman $\rho = -0.40$, $p = 0.04$ between $|\lambda_2(\bvec{T})|$ and this gap). Restoring belief updating reverses it, and the POMDP beats Myopic in all stores. As an ablation, this isolates belief updating as the component that keeps multi-period optimization from backfiring. To measure how much of the attainable information value the Bayesian filter captures, \ref{app:info_ladder} reports an information-ladder ablation. On synthetic days drawn from the fitted model for store B6, where the true regime path is known, the Bayesian filter realizes about two thirds of the information value attainable by any causal policy, the remainder being intrinsic to overlapping regime emissions.

%Figure~\ref{fig:voi_mechanism}(b) places the realized gain against the relevant information bounds, using synthetic days from the fitted model for store~A so the true regime path is known. Because these days are generated under the fitted model with the regime path known, the ladder measures what the filter can extract within the model rather than an out-of-sample value. Four policies share the same action rule and differ only in the belief they condition on. Perfect information about the current regime, the expected value of perfect information (EVPI), would be worth about 20\% per day over the static-belief policy, but no causal policy can reach it. Knowing the previous hour's regime, propagated one step through $\bvec{T}$, is worth about 12\%, so half the EVPI is lost to one hour of regime mixing. The Bayesian filter captures about 8\%, two thirds of this attainable bound, and the shortfall reflects regime emissions too overlapping to pin down the regime from a single hour of demand.

\begin{figure}[htbp]
    \centering
    \includegraphics[width=0.48\textwidth]{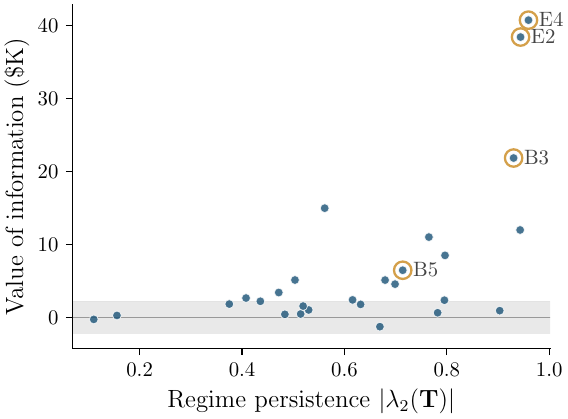}
    \caption{Per-store value of information, measured as the (Full POMDP $-$ Static-Belief) reward gap over the test period, against regime persistence $|\lambda_2(\bvec{T})|$. The shaded band marks $\pm 2\sigma$ of the evaluation noise. Circled points mark the stores where Static-Belief underperforms even Myopic and belief updating rescues the dynamic model.}
    \label{fig:voi_mechanism}
\end{figure}

% Figure~\ref{fig:voi_mechanism}(b) places the realized gain against the relevant information bounds, using synthetic days from the fitted model for store B6 so the true regime path is known. Because these days are generated under the fitted model with the regime path known, the ladder measures what the filter can extract within the model rather than an out-of-sample value. Four policies share the same action rule and differ only in the belief they condition on. Perfect information about the current regime, the expected value of perfect information (EVPI), would be worth about 20\% per day over the static-belief policy, but no causal policy can reach it. Knowing the previous hour's regime, propagated one step through $\bvec{T}$, is worth about 12\%, so half of the EVPI is lost to one hour of regime mixing. The Bayesian filter captures about 8\%, two thirds of this attainable bound, and the shortfall reflects regime emissions too overlapping to pin down the regime from a single hour of demand.

%%=========================================================================
% \subsection{Sensitivity Analysis}
% \label{sec:sensitivity}
%%=========================================================================

We assess the robustness of the POMDP policy by sweeping four economic parameters one at a time while holding all others at their base-case values. For each setting, backward induction is re-solved and the policy is re-evaluated on the same test sequence for store B6. Figure~\ref{fig:sensitivity_overview} plots the POMDP total reward as each parameter is varied from its baseline, expressed as a percentage change. The four parameters separate into three tiers by the reward range they induce. Revenue $p$ is largest (\$408K across its sweep), and it acts as a pure multiplier on a fixed operating plan. Reward moves almost exactly \$102K per \$5 of $p$ with deployment and fulfillment essentially unchanged, and the policy keeps serving nearly all demand even at $p = 10$, where the per fully-utilized driver-hour is negative, because the backlog and terminal penalties make under-serving costlier than the operating loss.

\begin{figure}[htbp]
    \centering
    \begin{subfigure}[b]{0.245\textwidth}\centering
        \includegraphics[width=\textwidth]{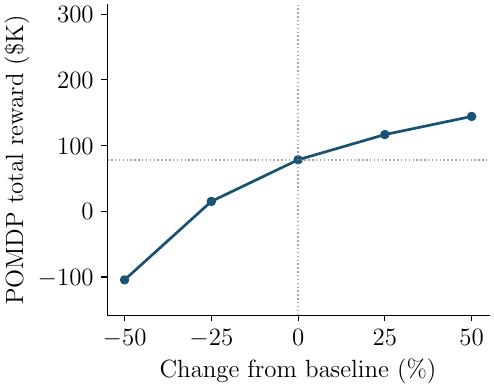}
        \caption{Driver capacity $v$}
    \end{subfigure}\hfill
    \begin{subfigure}[b]{0.245\textwidth}\centering
        \includegraphics[width=\textwidth]{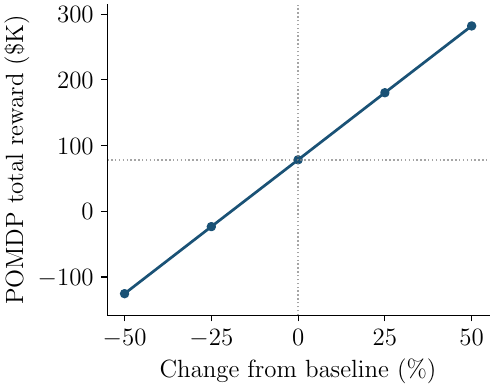}
        \caption{Revenue $p$}
    \end{subfigure}\hfill
    \begin{subfigure}[b]{0.245\textwidth}\centering
        \includegraphics[width=\textwidth]{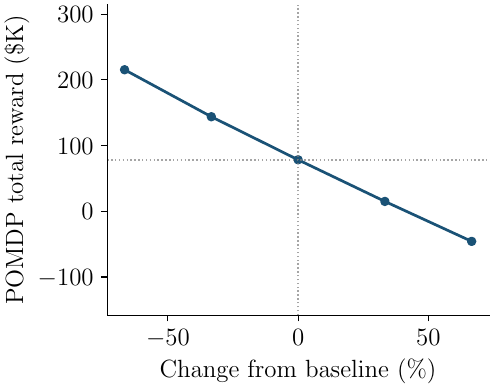}
        \caption{Driver wage $c_a$}
    \end{subfigure}\hfill
    \begin{subfigure}[b]{0.245\textwidth}\centering
        \includegraphics[width=\textwidth]{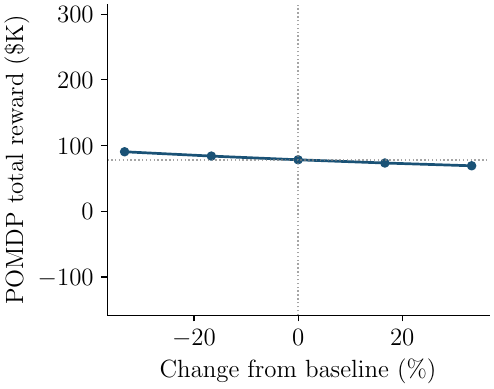}
        \caption{Backlog penalty $c_b$}
    \end{subfigure}
    \caption{Sensitivity of the POMDP total reward to one-at-a-time parameter perturbations (store B6, 115-day test period), each panel expressed as percentage change from baseline and drawn on a common reward scale.}
    \label{fig:sensitivity_overview}
\end{figure}

The two labor parameters $c_a$ and $v$ follow (\$261K and \$249K). Driver capacity $v$ is the most nonlinear of the four. At $v = 1$ the store needs almost twice the baseline fleet (21{,}523 versus 11{,}481 driver-hours) to hold service and operates at a loss, and gains saturate above the baseline. The remaining penalty parameter $c_b$ trail by an order of magnitude or more (\$21K) and is close to linear, with the policy trimming deployment as wages rise and holding the queue shorter as $c_b$ rises.

%%=========================================================================
\section{Multi-Store Driver Allocation via Lagrangian Relaxation}
\label{sec:multistore}
%%=========================================================================

The single-store POMDP (Sections~\ref{sec:pomdp}--\ref{sec:algorithm}) determines the optimal number of drivers for one store in isolation. In practice, a platform may operate $L$ stores in the same city, drawing from a shared pool of $N_t$ available drivers each period. We index stores by $\ell \in \{1, \ldots, L\}$. Each store $\ell$ has its own parameters ($\bvec{T}_\ell$, $\mu_{\ell,d,h}$, $f_{\ell,k}$) and policy $\pi_\ell$, all computed as in Sections~\ref{sec:pomdp}--\ref{sec:demand_estimation}. The network-level problem couples all stores through a capacity constraint:
\begin{equation}
    \max_{\{a_{\ell,t}\}_{\ell=1}^L} \;\; \sum_{\ell=1}^{L} \mathbb{E}\!\left[\sum_{t=1}^{T} r_\ell(x_{\ell,t},\, s_{\ell,t},\, a_{\ell,t})\right]
    \quad \text{s.t.} \quad \sum_{\ell=1}^{L} a_{\ell,t} \leq N_t, \quad \forall\, t,
    \label{eq:network_problem}
\end{equation}
where $a_{\ell,t}$ is the number of drivers allocated to store $\ell$ in period $t$, and each store $\ell$ has its own POMDP state $(s_{\ell,t}, \bvec{b}_{\ell,t})$ evolving independently conditional on the allocation. Solving this jointly is intractable. The state space is the product of all per-store state spaces, and the action space is combinatorial. We instead derive a per-store \emph{priority index}, a scalar that captures how urgently each store needs its next driver in its current state, and allocate drivers greedily by index each period. The Lagrangian relaxation provides the formal framework for constructing this index and bounding the suboptimality of the resulting policy.

\subsection{Lagrangian Decomposition}

We relax the coupling constraint by introducing a Lagrange multiplier $\lambda \geq 0$, interpreted as the \emph{shadow price per driver}. Attaching $\lambda$ to each period's capacity constraint in~\eqref{eq:network_problem} and rearranging:
\begin{align}
    L(\lambda) &= \max_{\{a_{\ell,t}\}} \;\; \sum_{\ell=1}^{L} \mathbb{E}\!\left[\sum_{t=1}^{T} r_\ell(x_{\ell,t},\, s_{\ell,t},\, a_{\ell,t})\right] + \lambda \mathbb{E} \left[\sum_{t=1}^{T}\!\left(N_t - \sum_{\ell=1}^{L} a_{\ell,t}\right)\right] \nonumber \\
    &= \max_{\{a_{\ell,t}\}} \;\; \sum_{\ell=1}^{L} \mathbb{E}\!\left[\sum_{t=1}^{T} \bigl(r_\ell(x_{\ell,t},\, s_{\ell,t},\, a_{\ell,t}) - \lambda \, a_{\ell,t}\bigr)\right] + \lambda \sum_{t=1}^{T} N_t \nonumber \\
    &= \sum_{\ell=1}^{L} \underbrace{\max_{\{a_{\ell,t}\}} \;\; \mathbb{E}\!\left[\sum_{t=1}^{T} \bigl(r_\ell(x_{\ell,t},\, s_{\ell,t},\, a_{\ell,t}) - \lambda \, a_{\ell,t}\bigr)\right]}_{L_\ell(\lambda)} \;+\; \lambda \sum_{t=1}^{T} N_t.
    \label{eq:lagrangian}
\end{align}
Therefore, the penalty $-\lambda\,a_{\ell,t}$ is absorbed into each store's reward. The stores share no state or constraint and each maximization is independent, so the last term can be disregarded. Each per-store subproblem $L_\ell(\lambda)$ is a POMDP identical to the formulation in Section~\ref{sec:pomdp}, with a modified single-period reward:
\begin{equation}
    r_\ell^{\lambda}(x, s, a) = r_\ell(x, s, a) - \lambda \, a,
    \label{eq:reward_lambda}
\end{equation}
which replaces the driver wage $c_a$ with the augmented cost $c_a + \lambda$. That is, the single-store methodology is unchanged apart from the higher per-driver cost. As $\lambda$ increases, drivers become more expensive and the store's optimal policy becomes more conservative, requesting fewer drivers. As $\lambda$ decreases toward zero, the store behaves as if drivers were free beyond their base wage and requests the optimal number of drivers obtained in Section \ref{sec:pomdp}.

By solving the single-store POMDP at different values of $\lambda$, we can trace out how each store's optimal driver count $a^*(\lambda)$ responds to the shadow price. For a given state $(s, \bvec{b})$, the POMDP might prescribe $a^* = 5$ drivers when $\lambda = 0$ but only $a^* = 3$ when $\lambda = 20$. The values of $\lambda$ at which $a^*$ drops by one driver are precisely the breakpoints that define the priority index, as we formalize below. From this perspective, $\lambda$ is not a quantity we need to compute for the final policy. Rather, it is a parameter we vary to extract each store's willingness to pay for each marginal driver.

\subsection{Indexability}
\label{sec:indexability}

The Whittle-index policy \citep{whittle1988restless} is, at heart, a \emph{greedy} heuristic. It assigns each marginal driver-slot a scalar index and fills the pool top-down. Greedy allocation is generally suboptimal under coupling constraints. Yet, if the single-store subproblem is \emph{indexable}, then this greedy ranking rule is exactly optimal for the Lagrangian-relaxed problem~\eqref{eq:lagrangian}, and its suboptimality on the original constrained problem~\eqref{eq:network_problem} is controlled by the duality gap. The relaxation~\eqref{eq:lagrangian} decouples across stores, so at any price $\lambda$ the relaxed optimum has each store choose $a^*_\ell(\lambda)$ independently. Ranking all marginal drivers by their breakpoints and filling the pool down to the clearing price therefore reproduces these per-store optima, which is exactly the optimum of the relaxation \citep{whittle1988restless, weber1990index}.

Formally, the single-store problem is indexable if, for every state $(s, \bvec{b})$, the optimal driver count $a^*(\lambda)$ is non-increasing in the shadow price $\lambda$. This is the natural economic requirement that a store should never request \emph{more} drivers when they become more expensive. If indexability holds, then as we sweep $\lambda$ upward, the optimal action decreases monotonically, yielding a unique breakpoint $\lambda$ at which the store gives up its next driver. This breakpoint is the index.

\begin{restatable}[Indexability of the Single-Store POMDP]{theorem}{thmIndexability}
\label{prop:indexability}
The single-store POMDP with reward~\eqref{eq:reward_lambda} is indexable. That is, for every state $(s, \bvec{b})$, the optimal driver count $a^*(\lambda)$ is non-increasing in the shadow price $\lambda$.
\end{restatable}

Proof in \ref{app:proofs}.

Theorem~\ref{prop:indexability} is the structural foundation of the Whittle index policy. We establish indexability by bounding $D_t(s, \bvec{b}; \lambda) = \mathbb{E}\bigl[\sum_{\tau=t}^{T} a^*_\tau\bigr]$, the expected total driver-hours consumed under the optimal policy from period $t$ onward. The argument proceeds by backward induction on the cumulative driver-hours functional $D_t$, using parts (ii) and (iii) of Theorem~\ref{prop:backlog_monotone} to bound how the optimal action moves under a unit and a $v$-unit shift in backlog. The proof leans on the exogenous-regime structure (A3). Because staffing does not move the demand regime, the belief path is action-independent and comparisons across shadow prices reduce to the backlog dimension, which  Theorem~\ref{prop:backlog_monotone} controls. When the action drives the latent state \citep{MeshramKaza2025}, this reduction is unavailable and indexability remains open.

The economic intuition behind the bounded $v$-step is that one order consumes at most $1/v$ of a driver-period if it is served, and zero if eventually dropped. Because an order cannot be served more than once, an additional $v$ units of backlog can never require more than one additional driver-period over the remaining horizon.

\begin{remark}
The Lagrangian relaxation also provides an upper bound on the optimal value of the network problem~\eqref{eq:network_problem}. Define the optimal global multiplier as
\begin{equation}
    \lambda^* = \arg\min_{\lambda \geq 0} \; L(\lambda).
    \label{eq:dual}
\end{equation}
Since $L(\lambda)$ is convex in $\lambda$ (as the pointwise maximum of affine functions), this can be solved by bisection or subgradient methods. The dual value satisfies $L(\lambda^*) \geq$ the optimal value of~\eqref{eq:network_problem}, so the gap $L(\lambda^*) - V^{\text{policy}}$ for any feasible policy bounds its suboptimality. This bound is useful for evaluating the quality of the index policy but plays no role in the allocation procedure itself.
\end{remark}

\subsection{Priority Index and Online Allocation}

Indexability guarantees that for each store $\ell$ in each state $(s_{\ell,t}, \bvec{b}_{\ell,t})$, the optimal action $a^*_\ell(\lambda)$ is a non-increasing staircase in $\lambda$. Because the action is integer-valued and a store may request multiple drivers, this staircase may have several steps. We define a \emph{driver priority index} for each marginal driver: the $j$-th index of store $\ell$ is the breakpoint
\begin{equation}
    \Lambda_\ell^{(j)}(s_{\ell,t}, \bvec{b}_{\ell,t}) = \inf\bigl\{\lambda \geq 0 : a^*_\ell(\lambda) < j\bigr\}, \quad j = 1, \ldots, a^*_\ell(0),
    \label{eq:index}
\end{equation}
that is, the maximum price per driver at which store $\ell$ still demands at least $j$ drivers. By indexability, $\Lambda_\ell^{(1)} \geq \Lambda_\ell^{(2)} \geq \cdots$, so each successive driver is valued less than the previous one. Intuitively, $\Lambda_\ell^{(j)}$ is the willingness to pay for the $j$-th driver. Stores with high backlog, or whose beliefs concentrate on a high-demand regime, will have high indices. The response of the index to backlog is a consequence of the structural results of Section \ref{sec:structural}, which carry over to the Lagrangian-augmented problem at every shadow price.

\begin{restatable}[]{corollary}{corIndexMono}
\label{cor:index_monotone}
For every store $\ell$, period $t$, belief $\bvec{b} \in \Delta^{K-1}$, and driver rank $j$, the index $\Lambda_\ell^{(j)}(s, \bvec{b})$ is non-decreasing in the backlog $s$. Moreover, $\Lambda_\ell^{(j+1)}(s+v, \bvec{b}) \le \Lambda_\ell^{(j)}(s, \bvec{b})$ for every $s \ge 0$ and every rank $j$.
\end{restatable}

Proof in \ref{app:proofs}. 

Congestion therefore raises a store's bid for every marginal driver, so the ranking responds to backlog in the direction the allocation requires. The second inequality bounds the rise. Since one driver clears $v$ orders per hour, $v$ additional backlogged orders never lift the store's bid for its next marginal driver above its current bid for the previous one, so a congested store climbs the ranking gradually rather than seizing the pool.

The indices are computed offline. For each store $\ell$, we solve the single-store POMDP for a grid of $\lambda$ values $0 = \lambda_0 < \lambda_1 < \cdots < \lambda_G$ and record $a^*(\lambda_i)$ at each state $(s, \bvec{b})$. The indices $\Lambda_\ell^{(j)}(s, \bvec{b})$ are the breakpoints where the optimal action drops below $j$, interpolated between adjacent grid values.

Figure \ref{fig:indexability} illustrates this structure empirically on three stores in the same city (City E). At a representative snapshot (Sunday 13:00, stationary belief), the optimal staffing $a^*(\lambda)$ is a non-increasing staircase in the shadow price for every store, consistent with Theorem~\ref{prop:indexability}. Each downward step is a breakpoint $\Lambda^{(j)}$ in the sense of Equation~\eqref{eq:index}. Panel \ref{fig:indexability_b} shows the same staircase at increasing backlog levels. A congested store's staircase shifts outward, so it out-bids its peers for every marginal driver, consistent with Corollary \ref{cor:index_monotone}.

\begin{figure}[htbp]
    \centering
    \begin{subfigure}[b]{0.45\textwidth}
        \centering
        \includegraphics[width=\textwidth]{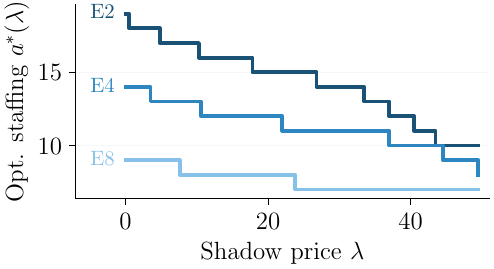}
        \caption{Three-store comparison (stationary belief, $s = 0$).}
        \label{fig:indexability_a}
    \end{subfigure}
    \hfill
    \begin{subfigure}[b]{0.45\textwidth}
        \centering
        \includegraphics[width=\textwidth]{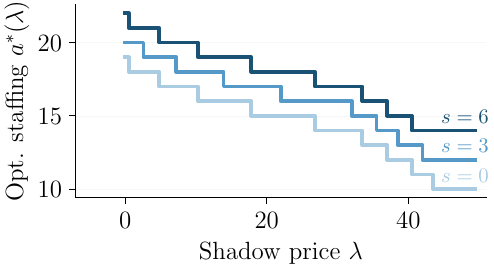}
        \caption{Store E2 (stationary belief, different backlogs)}
        \label{fig:indexability_b}
    \end{subfigure}
    \caption{Empirical indexability
    (Theorem~\ref{prop:indexability}) computed on Sunday 13:00.}
    \label{fig:indexability}
\end{figure}

% \begin{figure}[htbp]
%     \centering
%     \includegraphics[width=\textwidth]{figures/fig_indexability.pdf}
%     \caption{. (a)~
%     (b)~}
%     \label{fig:indexability}
% \end{figure}

These indices respond to the demand belief, making allocation anticipatory of demand surges. The indices of a store's first drivers are set by the participation margin, the price at which serving at all becomes worthwhile, and they stay uniformly high at every belief. The indices of the marginal drivers, by contrast, rise steeply as the belief shifts toward the high-demand regime. A store that believes a surge is coming raises its bid for extra capacity before any backlog materializes.

Given precomputed index tables, the online allocation each period proceeds as follows:
\begin{enumerate}[nosep]
    \item Each store $\ell$ reports its current state $(s_{\ell,t}, \bvec{b}_{\ell,t})$ and looks up its index values $\Lambda_\ell^{(1)}, \Lambda_\ell^{(2)}, \ldots$
    \item All driver-slots across all stores are pooled and ranked by index value in descending order.
    \item Drivers are allocated greedily down the ranked list until the pool of $N_t$ drivers is exhausted.
\end{enumerate}

This construction follows the same logic as the Whittle index for restless multi-armed bandits \citep{whittle1988restless}. In the classical Whittle framework, each ``arm'' (store) has a binary action (activate or rest), and the index is the subsidy for passivity that makes the arm indifferent. Our formulation generalizes this to integer-valued actions, producing multiple indices per store, one per marginal driver. In the binary-action case, \citet{weber1990index} show that the Whittle index policy is asymptotically optimal for the constrained problem (and not just the relaxed version) as the number of arms and the budget grow proportionally. Extending this result to multi-action arms remains an open question, and we leave this for future work.

To illustrate the full procedure, consider a network with three stores and $N_t = 6$ available drivers. In the precomputation phase, we solve each store's single-store POMDP repeatedly at different values of $\lambda$. Suppose that at the current period, the three stores are in states that produce the following index tables:

\begin{center}
\small
\begin{tabular}{lcccc}
\toprule
& $\Lambda^{(1)}$ & $\Lambda^{(2)}$ & $\Lambda^{(3)}$ & $\Lambda^{(4)}$ \\
\midrule
Store A & 80 & 50 & 20 & 5 \\
Store B & 70 & 35 & 10 & -- \\
Store C & 60 & 15 & -- & -- \\
\bottomrule
\end{tabular}
\end{center}

\noindent For example, Store~A's POMDP prescribes $a^* = 4$ drivers when $\lambda = 0$, but only $a^* = 3$ when $\lambda > 5$, $a^* = 2$ when $\lambda > 20$, and so on. Its fourth driver is worth only $\Lambda^{(4)} = 5$, while its first driver is worth $\Lambda^{(1)} = 80$. To allocate the $N_t = 6$ drivers, we pool all driver-slots across all stores and rank them by index value. Ranking all slots by index gives 80, 70, 60, 50, 35, 20, 15, 10, 5. The first six go to A’s 1st–3rd, B’s 1st–2nd, and C’s 1st, so Store A receives 3 drivers, Store B 2, and Store C 1. The implicit clearing price falls between 15 and 20, and A’s 4th and C’s 2nd drivers are priced out.

%%=========================================================================
\subsection{Multi-Store Illustrative Allocation}
\label{sec:multi_store_results}
%%=========================================================================

To illustrate this policy in operation, Figure \ref{fig:index_policy_demo} traces two consecutive test days for the three City~E stores of Section \ref{sec:indexability} (E2, E4, E8) sharing a pool of $N = 36$ drivers. Beliefs are filtered from realized demand exactly as in the single-store evaluation, and each hour the pool is allocated by priority index. The allocation is anticipatory rather than backlog-reactive as capacity follows each store's filtered belief hour by hour, before any queue has formed. The clearest signature is the pair E4 and E8. Their calendar baselines over these weekdays differ markedly (17.0 versus 12.2 orders per hour), yet realized demand over the two days shown was nearly identical (13.8 versus 14.3), and the index policy, tracking each store's filtered belief rather than the calendar ranking, allocated the two stores nearly identical driver totals while their hourly splits diverged whenever their beliefs did. The largest store, E2, draws the most drivers overall, but during the system-wide evening peak the shared pool binds and the policy triages by priority, shifting capacity toward whichever stores currently believe themselves closest to a surge.

\begin{figure}[htbp]
    \centering
    \includegraphics[width=0.9\textwidth]{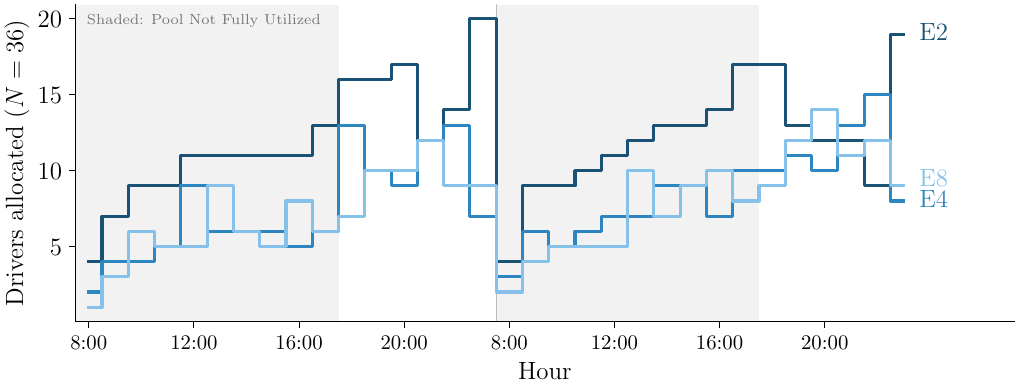}
    \caption{The index policy in action over two consecutive test days, with three stores ($E_2$, $E_4$, $E_8$) drawing on a shared pool of $N = 36$ drivers.}
    \label{fig:index_policy_demo}
\end{figure}

%%=========================================================================
\section{Conclusions}
\label{sec:conclusions}
%%=========================================================================

We studied driver staffing for quick-commerce dark stores and made two contributions. The first is a single-store POMDP in which the demand regime is latent and the value of tracking it in real time can be measured. The model is dynamic in two ways. It carries unfulfilled orders forward as backlog, so a backlog-aware dynamic program is needed to avoid the service failures that a single-period rule lets compound. It also keeps a belief over the demand regime and updates it each hour. On a 27-store panel, real-time belief updating adds 10.9\% over the same dynamic program with a fixed belief, or about \$188K. The gain is not spread evenly. It falls on the stores whose regimes are most persistent, where it can reach 49\% of a store's reward. Because persistence can be measured at estimation time, the firm can tell in advance which stores reward real-time learning.

We then prove that the single-store problem is indexable, so as the shadow price on a driver rises the optimal number of drivers never rises and each marginal driver has a well-defined priority index, one that rises with the store’s backlog. Indexability turns the network problem into a ranking rule that pools the shared fleet, ranks every store's marginal driver by its index, and fills the pool from the top down. We demonstrate this index policy on a shared fleet, where it moves capacity toward the stores with the highest current need before any backlog forms. The greedy allocation is optimal for the Lagrangian relaxation of the network problem, and the duality gap bounds how far it can fall short on the constrained problem. How tight that gap is in practice, and whether the asymptotic optimality known for binary-action restless bandits carries over to many actions, are questions we leave for future work.

\textit{Statement: The authors used Claude to document and revise code, and to improve writing presentation. The authors reviewed and edited the content and take full responsibility for the content of the article.}

{\setlength{\bibsep}{1pt plus 0.4pt}\renewcommand{\bibfont}{}
\bibliographystyle{apalike}
\bibliography{bibliography}}

@article{treharne2002adaptive,
  title={Adaptive inventory control for nonstationary demand and partial information},
  author={Treharne, James T and Sox, Charles R},
  journal={Management Science},
  volume={48},
  number={5},
  pages={607--624},
  year={2002},
  publisher={INFORMS}
}

@misc{europarlGigEconomy,
  title        = {Gig economy: how the EU improves platform workers' rights},
  howpublished = {\url{https://www.europarl.europa.eu/topics/en/article/20190404STO35070/gig-economy-how-the-eu-improves-platform-workers-rights}},
  note         = {Accessed: 2025-03-26},
  author       = {{European Parliament}},
  year         = {2019}
}

@misc{globenewswire2022,
  author = "GlobeNewswire",
  title = "Global Online Grocery Market 2022 to 2030 - Size, Share, Trends, Analysis Report",
  year = 2022,
  howpublished = "\url{https://www.globenewswire.com/en/news-release/2022/06/13/2461011/28124/en/Global-Online-Grocery-Market-2022-to-2030-Size-Share-Trends-Analysis-Report.html}",
  note = "[Accessed: 26-June-2023]"
}

@article{rabiner1989tutorial,
  title={A tutorial on hidden Markov models and selected applications in speech recognition},
  author={Rabiner, Lawrence R},
  journal={Proceedings of the IEEE},
  volume={77},
  number={2},
  pages={257--286},
  year={1989},
  publisher={IEEE}
}

@article{olbrys2022approximate,
  title={Approximate entropy and sample entropy algorithms in financial time series analyses},
  author={Olbrys, Joanna and Majewska, Elzbieta},
  journal={Procedia Computer Science},
  volume={207},
  pages={255--264},
  year={2022},
  publisher={Elsevier}
}

@article{LIANG2026100240,
title = {A Markov decision process framework for order dispatching in on-demand delivery services},
journal = {Multimodal Transportation},
volume = {5},
number = {1},
pages = {100240},
year = {2026},
issn = {2772-5863},
doi = {https://doi.org/10.1016/j.multra.2025.100240},
url = {https://www.sciencedirect.com/science/article/pii/S2772586325000541},
author = {Jian Liang and Jintao Ke}
}

@article{YildizSavelsbergh2019,
title = {Service and capacity planning in crowd-sourced delivery},
journal = {Transportation Research Part C: Emerging Technologies},
volume = {100},
pages = {177-199},
year = {2019},
issn = {0968-090X},
doi = {https://doi.org/10.1016/j.trc.2019.01.021},
url = {https://www.sciencedirect.com/science/article/pii/S0968090X18311513},
author = {Baris Yildiz and Martin Savelsbergh}
}

@article{DaiLiu2020,
  author  = {Dai, Hongyan and Liu, Peng},
  title   = {Workforce planning for O2O delivery systems with crowdsourced drivers},
  journal = {Annals of Operations Research},
  year    = {2020},
  volume  = {291},
  number  = {1},
  pages   = {219--245},
  month   = {aug},
  doi     = {10.1007/s10479-019-03135-z},
  url     = {https://doi.org/10.1007/s10479-019-03135-z},
  issn    = {1572-9338}}

@article{UlmerThomasMattfeld2019,
title = {Preemptive depot returns for dynamic same-day delivery},
journal = {EURO Journal on Transportation and Logistics},
volume = {8},
number = {4},
pages = {327-361},
year = {2019},
issn = {2192-4376},
doi = {https://doi.org/10.1007/s13676-018-0124-0},
url = {https://www.sciencedirect.com/science/article/pii/S2192437620300479},
author = {Marlin W. Ulmer and Barrett W. Thomas and Dirk C. Mattfeld}
}

@article{UlmerThomasCampbellWoyak2021,
author = {Ulmer, Marlin W. and Thomas, Barrett W. and Campbell, Ann Melissa and Woyak, Nicholas},
title = {The Restaurant Meal Delivery Problem: Dynamic Pickup and Delivery with Deadlines and Random Ready Times},
journal = {Transportation Science},
volume = {55},
number = {1},
pages = {75-100},
year = {2021},
doi = {10.1287/trsc.2020.1000}}

@article{ZehtabianLarsenWohlk2022,
title = {Estimation of the arrival time of deliveries by occasional drivers in a crowd-shipping setting},
journal = {European Journal of Operational Research},
volume = {303},
number = {2},
pages = {616-632},
year = {2022},
issn = {0377-2217},
doi = {https://doi.org/10.1016/j.ejor.2022.02.050},
url = {https://www.sciencedirect.com/science/article/pii/S0377221722001710},
author = {Shohre Zehtabian and Christian Larsen and Sanne Wøhlk}
}

@article{SavelsberghUlmer2024,
  author  = {Savelsbergh, Martin and Ulmer, Marlin W.},
  title   = {Challenges and opportunities in crowdsourced delivery planning and operations—an update},
  journal = {Annals of Operations Research},
  year    = {2024},
  volume  = {343},
  pages   = {639--661},
  doi     = {10.1007/s10479-024-06249-1},
  publisher = {Springer}
}

@article{BehrendtBehrendtWang2024,
title = {Task assignment, pricing, and capacity planning for a hybrid fleet of centralized and decentralized couriers},
journal = {Transportation Research Part C: Emerging Technologies},
volume = {160},
pages = {104533},
year = {2024},
issn = {0968-090X},
doi = {https://doi.org/10.1016/j.trc.2024.104533},
url = {https://www.sciencedirect.com/science/article/pii/S0968090X24000548},
author = {Adam Behrendt and Martin Behrendt and He Wang}
}

@misc{MeshramKaza2025,
  title={Lagrangian Relaxation for Multi-Action Partially Observable Restless Bandits: Heuristic Policies and Indexability},
  author={Meshram, Rahul and Kaza, Kesav},
  year={2025},
  howpublished={arXiv preprint arXiv:2509.00415}
}

@article{ArslanAgatzKroonZuidwijk2019,
  author = {Alp M. Arslan and Niels Agatz and Leo Kroon and Rob Zuidwijk},
  title = {Crowdsourced Delivery—A Dynamic Pickup and Delivery Problem with Ad Hoc Drivers},
  journal = {Transportation Science},
  volume = {53},
  number = {1},
  pages = {222--235},
  year = {2019},
  doi = {10.1287/trsc.2017.0803}
}

@article{AfecheLiuMaglaras2023,
    author = {Philipp Af{\`e}che and Zhe Liu and Costis Maglaras},
    title = {Ride-Hailing Networks with Strategic Drivers: The Impact of Platform Control Capabilities on Performance},
    journal = {Manufacturing \& Service Operations Management},
    volume = {25},
    number = {5},
    pages = {1890--1908},
    year = {2023},
    doi = {10.1287/msom.2023.1221}
  }

@article{weber1990index,
    author = {Weber, Richard R. and Weiss, Gideon},
    title = {On an Index Policy for Restless Bandits},
    journal = {Journal of Applied Probability},
    volume = {27},
    number = {3},
    pages = {637--648},
    year = {1990}
  }

@article{whittle1988restless,
  title={Restless bandits: Activity allocation in a changing world},
  author={Whittle, Peter},
  journal={Journal of applied probability},
  volume={25},
  number={A},
  pages={287--298},
  year={1988},
  publisher={Cambridge University Press}
}

@article{sondik1978optimal,
  title={The optimal control of partially observable Markov processes over the infinite horizon: Discounted costs},
  author={Sondik, Edward J},
  journal={Operations research},
  volume={26},
  number={2},
  pages={282--304},
  year={1978},
  publisher={INFORMS}
}

@misc{medianama_10min_rollback,
    author       = {{MediaNama}},
    title        = {Why the 10-Minute Delivery Model Is No Longer Feasible},
    year         = {2026},
    howpublished = {\url{https://www.medianama.com/2026/01/223-rollback-of-10-minute-delivery-quick-commerce/}},
    note         = {Accessed June 2026}
  }

@misc{mckinsey_ecommerce_delivery,
    author       = {{McKinsey \& Company}},
    title        = {What Do {US} Consumers Want from E-Commerce Deliveries?},
    year         = {2021},
    howpublished = {\url{https://www.mckinsey.com/industries/logistics/our-insights/what-do-us-consumers-want-from-e-commerce-deliveries}},
    note         = {Accessed June 2026}
  }

@article{SongZipkin1993,
  author = {Song, Jing-Sheng and Zipkin, Paul},
  title = {Inventory Control in a Fluctuating Demand Environment},
  journal = {Operations Research},
  volume = {41},
  number = {2},
  pages = {351--370},
  year = {1993},
  doi = {10.1287/opre.41.2.351}
}

@article{LariviereePorteus1999,
  author = {Lariviere, Martin A. and Porteus, Evan L.},
  title = {Stalking Information: {Bayesian} Inventory Management with Unobserved Lost Sales},
  journal = {Management Science},
  volume = {45},
  number = {3},
  pages = {346--363},
  year = {1999},
  doi = {10.1287/mnsc.45.3.346}
}

@phdthesis{Hawkins2003,
  author = {Hawkins, Jeffrey T.},
  title = {A Lagrangian Decomposition Approach to Weakly Coupled Dynamic Optimization Problems and Its Applications},
  school = {Massachusetts Institute of Technology},
  year = {2003}
}

@article{HodgeGlazebrook2015,
  author = {Hodge, David J. and Glazebrook, Kevin D.},
  title = {On the Asymptotic Optimality of Greedy Index Heuristics for Multi-Action Restless Bandits},
  journal = {Advances in Applied Probability},
  volume = {47},
  number = {3},
  pages = {652--667},
  year = {2015},
  doi = {10.1239/aap/1444308876}
}

@article{Scarf1959,
  author = {Scarf, Herbert},
  title = {Bayes Solutions of the Statistical Inventory Problem},
  journal = {The Annals of Mathematical Statistics},
  volume = {30},
  number = {2},
  pages = {490--508},
  year = {1959}
}

@article{Azoury1985,
  author = {Azoury, Katy S.},
  title = {Bayes Solution to Dynamic Inventory Models under Unknown Demand Distribution},
  journal = {Management Science},
  volume = {31},
  number = {9},
  pages = {1150--1160},
  year = {1985},
  doi = {10.1287/mnsc.31.9.1150}
}

@article{Lovejoy1987,
  author = {Lovejoy, William S.},
  title = {Some Monotonicity Results for Partially Observed Markov Decision Processes},
  journal = {Operations Research},
  volume = {35},
  number = {5},
  pages = {736--743},
  year = {1987},
  doi = {10.1287/opre.35.5.736}
}

@article{PapadimitriouTsitsiklis1999,
  author = {Papadimitriou, Christos H. and Tsitsiklis, John N.},
  title = {The Complexity of Optimal Queuing Network Control},
  journal = {Mathematics of Operations Research},
  volume = {24},
  number = {2},
  pages = {293--305},
  year = {1999},
  doi = {10.1287/moor.24.2.293}
}

@article{NinoMora2001,
  author = {Ni{\~n}o-Mora, Jos{\'e}},
  title = {Restless Bandits, Partial Conservation Laws and Indexability},
  journal = {Advances in Applied Probability},
  volume = {33},
  number = {1},
  pages = {76--98},
  year = {2001}
}

@article{GansKooleMandelbaum2003,
  author = {Gans, Noah and Koole, Ger and Mandelbaum, Avishai},
  title = {Telephone Call Centers: Tutorial, Review, and Research Prospects},
  journal = {Manufacturing \& Service Operations Management},
  volume = {5},
  number = {2},
  pages = {79--141},
  year = {2003},
  doi = {10.1287/msom.5.2.79.16071}
}

@article{Taylor2018,
  author = {Taylor, Terry A.},
  title = {On-Demand Service Platforms},
  journal = {Manufacturing \& Service Operations Management},
  volume = {20},
  number = {4},
  pages = {704--720},
  year = {2018}
}

@article{CachonDanielsLobel2017,
  author = {Cachon, G{\'e}rard P. and Daniels, Kaitlin M. and Lobel, Ruben},
  title = {The Role of Surge Pricing on a Service Platform with Self-Scheduling Capacity},
  journal = {Manufacturing \& Service Operations Management},
  volume = {19},
  number = {3},
  pages = {368--384},
  year = {2017},
  doi = {10.1287/msom.2017.0618}
}

@article{baum1970maximization,
  author = {Baum, Leonard E. and Petrie, Ted and Soules, George and Weiss, Norman},
  title = {A Maximization Technique Occurring in the Statistical Analysis of Probabilistic Functions of {Markov} Chains},
  journal = {The Annals of Mathematical Statistics},
  volume = {41},
  number = {1},
  pages = {164--171},
  year = {1970}
}

@book{topkis1998supermodularity,
  author = {Topkis, Donald M.},
  title = {Supermodularity and Complementarity},
  publisher = {Princeton University Press},
  year = {1998}
}

@misc{eu_platform_work_directive_2024,
  author = {{European Union}},
  title = {Directive ({EU}) 2024/2831 of the {European} {Parliament} and of the {Council} on Improving Working Conditions in Platform Work},
  year = {2024},
  howpublished = {Official Journal of the European Union},
  note = {\url{https://eur-lex.europa.eu/eli/dir/2024/2831/oj/eng}}
}

@misc{handelsblatt_qcommerce_labor,
  author = {{Handelsblatt}},
  title = {Lieferdienste: Arbeitsbedingungen bei {Gorillas}, {Flink} und {Getir} werden zum Wettbewerbsfaktor},
  year = {2021},
  howpublished = {\url{https://www.handelsblatt.com/unternehmen/handel-konsumgueter/lieferdienste-proteste-der-gorillas-kuriere-wirken-arbeitsbedingungen-werden-zum-wettbewerbsfaktor/27515512.html}},
  note = {Accessed June 2026}
}

@misc{bhrrc_glovo_employees,
  author = {{Business and Human Rights Resource Centre}},
  title = {Spain: {Glovo} Riders in {Spain} to Become Employees},
  year = {2024},
  howpublished = {\url{https://www.business-humanrights.org/en/latest-news/espagne-glovo-a-salarier-ses-livreurs-a-domicile/}},
  note = {Accessed June 2026}
}

@article{MilgromSegal2002,
  title={Envelope theorems for arbitrary choice sets},
  author={Milgrom, Paul and Segal, Ilya},
  journal={Econometrica},
  volume={70},
  number={2},
  pages={583--601},
  year={2002},
  publisher={Wiley Online Library}
}

@article{VieiraMendonca2025,
  author  = {Vieira, Tiago and Mendon\c{c}a, Pedro},
  title   = {The Times, Are They Changing? {E}xamining Platform Companies' Chameleonic
             Labour Process as a Response to the {Spanish} {Ley} {Rider}},
  journal = {Socio-Economic Review},
  volume  = {23},
  number  = {2},
  pages   = {877--898},
  year    = {2025},
  doi     = {10.1093/ser/mwae066}
}

@misc{Aranguiz2021,
  author       = {Aranguiz, Ane},
  title        = {Platforms Put a Spoke in the Wheels of {Spain's} `Riders' Law'},
  howpublished = {Social Europe},
  year         = {2021},
  note         = {2 September 2021},
  url          = {https://www.socialeurope.eu/platforms-put-a-spoke-in-the-wheels-of-spains-riders-law}
}

@article{SteinkerHobergThonemann2017,
  author  = {Steinker, Sebastian and Hoberg, Kai and Thonemann, Ulrich W.},
  title   = {The Value of Weather Information for E-Commerce Operations},
  journal = {Production and Operations Management},
  volume  = {26},
  number  = {10},
  pages   = {1854--1874},
  year    = {2017},
  doi     = {10.1111/poms.12721}
}

@article{BadorfHoberg2020,
  author  = {Badorf, Florian and Hoberg, Kai},
  title   = {The Impact of Daily Weather on Retail Sales: An Empirical Study in Brick-and-Mortar Stores},
  journal = {Journal of Retailing and Consumer Services},
  volume  = {52},
  pages   = {101921},
  year    = {2020},
  doi     = {10.1016/j.jretconser.2019.101921}
}

@misc{Wang2026,
title={Partially Observable Restless Bandits for Age-Optimal Scheduling over Markov Channels},
author={Wang, Xijun and Gan, Shuying and Huang, Yanzhi and Zhao, Xiaoyu and Xu, Chao and Chen, Xiang},
year={2026},
howpublished={arXiv preprint arXiv:2605.21016}
}

@article{Fu2025,
author = {Fu, Jing and Zhang, Lele and Liu, Zhiyuan},
title = {A restless bandit model for dynamic ride matching with reneging travelers},
journal = {European Journal of Operational Research},
volume = {320},
number = {3},
pages = {581--592},
year = {2025}
}

@article{ArifogluOzekici2010,
author = {Arifo{\u{g}}lu, Kenan and {\"O}zekici, S{\"u}leyman},
title = {Optimal policies for inventory systems with finite capacity and partially observed Markov-modulated demand and supply processes},
journal = {European Journal of Operational Research},
volume = {204},
number = {3},
pages = {421--438},
year = {2010}
}

@article{liu2025relaxed,
  title={Relaxed indexability and index policy for partially observable restless bandits},
  author={Liu, Keqin},
  journal={Management Science},
  volume={71},
  number={12},
  pages={10106--10121},
  year={2025},
  publisher={INFORMS}
}

\newpage
\appendix

\section{Proofs}
\label{app:proofs}

This appendix collects the proofs of the structural and indexability results stated in Sections~\ref{sec:structural} and~\ref{sec:indexability}. The proofs are presented in the order in which the results appear in the main body, which also reflects the logical dependencies.

For ease of reference in the structural results that follow, we record the four conditions already introduced above under shorthand labels. \textbf{(A1)}~the net per-order margin $q := p - c_p$ is strictly positive, so each fulfilled order contributes positively to the reward; \textbf{(A2)}~the per-period backlog penalty dominates the margin, $c_b > q$, so an unfulfilled order is more costly than the margin a fulfilled order would generate; \textbf{(A3)}~the regime transition matrix $\bvec{T}$ is action-independent, as discussed in Section~\ref{sec:pomdp}; and \textbf{(A4)}~the action upper bound is non-binding at every backlog $s$ the finite-horizon recursion visits, so the optimal allocation is never constrained by $a_{\max}$. All four are stated as part of the model in Section~\ref{sec:pomdp} above and are not additional assumptions imposed at the proof stage; the labels (A1)--(A4) are used throughout the proofs below.

\begin{lemma}
\label{lem:continuation_vconcave}
Let $W : \mathbb{Z}_{\ge 0} \to \mathbb{R}$ be non-increasing with non-increasing $v$-step differences, that is $s \mapsto W(s+v) - W(s)$ is non-increasing. Then $G(y) := W(\max\{y,0\})$ is non-increasing in $y \in \mathbb{Z}$ and satisfies $G(z+1+v) - G(z+1) \le G(z+v) - G(z)$ for every $z \in \mathbb{Z}$.
\end{lemma}
\begin{proof}
$G$ is the composition of the non-decreasing $y \mapsto \max\{y,0\}$ with the non-increasing $W$, hence non-increasing. It remains to verify the $v$-step inequality, which we do by splitting $\mathbb{Z}$ into the three ranges on which the clamp $\max\{\cdot,0\}$ behaves differently. First, when $z \ge 0$, all four arguments are non-negative and the inequality is the non-increasing $v$-step difference of $W$. Second, when $-v \le z \le -1$, the two smaller arguments $z+1,z$ map to $0$ while $z+v,z+1+v \ge 0$, so the inequality reduces to non-increasingness of $W$. Third, when $z < -v$, all four values coincide and both sides vanish. The inequality therefore holds for every $z$.
\end{proof}

\begin{lemma}
\label{lem:translation}
Fix a period $t$ and a belief $\bvec{b}$, and assume (A3). For every $s \in \mathbb{Z}_{\ge 0}$ and every $a \in [a_{\min}, a_{\max}-1]$, the $Q$-function satisfies
\begin{equation}
Q_t(s+v,\,a+1) \;=\; Q_t(s,a) \;+\; \kappa, \qquad \kappa \;:=\; (q-c_b)\,v - c_a,
\label{eq:translation}
\end{equation}
with $\kappa < 0$ by (A2).
\end{lemma}

\begin{proof}
Compare the two state-action pairs $(s,a)$ and $(s+v,a+1)$. With $e := x+s$, moving to $(s+v,a+1)$ replaces the effective demand $e$ by $e+v$ and the capacity $va$ by $v(a+1)=va+v$, and the served quantity and next-period backlog transform by elementary algebra. For the served quantity, adding the same constant $v$ to both entries of a minimum shifts the minimum by $v$,
\[
\min\{e+v,\,v(a+1)\} \;=\; \min\{e+v,\,va+v\} \;=\; v + \min\{e,\,va\},
\]
so the served amount rises by exactly $v$. For the next-period backlog, expanding $v(a+1)=va+v$ cancels the added backlog against the added capacity,
\[
\bigl(e+v - v(a+1)\bigr)^+ \;=\; (e+v-va-v)^+ \;=\; (e - va)^+,
\]
so the carried-over backlog, and with it the continuation value, is unchanged. Moving from $(s,a)$ to $(s+v,a+1)$ therefore raises the served-revenue term $q\min\{e,va\}$ by $qv$, raises the backlog penalty $-c_b s$ by $-c_b v$, and raises the wage $-c_a a$ by $-c_a$, while the continuation value is unchanged because $\bvec{b}_{t+1}(x)$ is action-independent under (A3). Summing these increments gives \eqref{eq:translation} with $\kappa = (q-c_b)v - c_a$, and $\kappa < 0$ by (A2).
\end{proof}

\begin{lemma}
\label{lem:threshold}
Let $M, N : \mathbb{Z}_{\ge 0} \to \mathbb{R}$ and $\kappa \in \mathbb{R}$, and suppose there is a threshold $s^\circ \in \mathbb{Z}_{\ge 0} \cup \{+\infty\}$ with $M(s) = N(s)$ for $s < s^\circ$ and $M(s) = \kappa$ for $s \ge s^\circ$. If $N$ is non-increasing and, whenever $0 < s^\circ < \infty$, $N(s^\circ - 1) \ge \kappa$, then $M$ is non-increasing on $\mathbb{Z}_{\ge 0}$.
\end{lemma}

\begin{proof}
On $\{s < s^\circ\}$ the function $M = N$ is non-increasing, and on $\{s \ge s^\circ\}$ it is the constant $\kappa$. If $s^\circ \in \{0, +\infty\}$ one of these sets is empty and the claim is immediate. Otherwise the only place $M$ could rise is from $s^\circ - 1$ to $s^\circ$, and there $M(s^\circ - 1) = N(s^\circ - 1) \ge \kappa = M(s^\circ)$ by hypothesis. Hence $M$ is non-increasing throughout.
\end{proof}

\thmBacklog*

\begin{proof}[Proof of Theorem~\ref{prop:backlog_monotone}]
We fix $\bvec{b}$, suppress it, and write $V_t(s) := V_t(s,\bvec{b})$ and $Q_t(s,a) := Q_t(s,\bvec{b},a)$; by (A3) the posterior $\bvec{b}_{t+1}(x)$ is action-independent, so the suppression is harmless. We write $\Delta_s$ and $\Delta_a$ for the unit forward differences $\Delta_s\phi(s,a):=\phi(s+1,a)-\phi(s,a)$ and $\Delta_a\phi(s,a):=\phi(s,a+1)-\phi(s,a)$, and $\Delta_s\Delta_a$ for their composite cross-difference; the $v$-step difference is always written out.

The argument is a backward induction on $t$. To carry the induction through, we prove an invariant strictly stronger than the three claims:
\[
(H_t): \qquad s \mapsto V_t(s) \text{ is non-increasing, and its $v$-step difference } s \mapsto V_t(s+v) - V_t(s) \text{ is non-increasing.}
\]
Claims (i)--(iii) all follow from $(H_t)$ as shown below; the non-increasing $v$-step difference is the extra structure the induction needs but which the theorem does not report.

We argue by backward induction on $t$, proving $(H_t)$ for $t = T+1, T, \ldots, 1$. The base case is $t = T+1$, where the terminal value $V_{T+1}(s) = -c_{\mathrm{lost}}\,s$ is linear in $s$, hence non-increasing with a constant, and so non-increasing, $v$-step difference, and $(H_{T+1})$ holds. For the inductive step we assume $(H_{t+1})$ and establish $(H_t)$ in four pieces, each read off the two lemmas above. First, the continuation Lemma~\ref{lem:continuation_vconcave} together with the reward cross-difference makes $Q_t$ have increasing differences in $(s,a)$, which gives the monotone policy~(ii) by Topkis. Second, the translation identity of Lemma~\ref{lem:translation}, which compares $(s,a)$ with $(s+v,a+1)$ at the fixed cost $\kappa$, gives the bounded increase~(iii). Third, the strict decrease~(i) follows from $\Delta_s r \le q - c_b < 0$ by (A2). Fourth, a threshold argument packaged in Lemma~\ref{lem:threshold} shows the $v$-step difference of $V_t$ is non-increasing, the second half of $(H_t)$. Establishing all four proves $(H_t)$ and advances the induction one period, as recorded where the step closes below.

We first prove the monotone policy~(ii). By Lemma~\ref{lem:continuation_vconcave} applied to $W = V_{t+1}(\cdot,\,\bvec{b}_{t+1}(x))$, which by $(H_{t+1})$ is non-increasing with non-increasing $v$-step difference, the clamped continuation $G(y) := V_{t+1}(\max\{y,0\},\,\bvec{b}_{t+1}(x))$ inherits both properties for each $x$. We compute the cross-difference of $Q_t$ in two pieces. In the reward $r = q\min\{e,va\} - c_a a - c_b s$ with $e := x+s$, the wage $-c_a a$ depends on $a$ alone and the penalty $-c_b s$ on $s$ alone, so each is killed by the difference in the other variable and only the served-margin term $q\min\{e,va\}$ survives the cross-difference. The reward piece uses $\min\{e+1,m\} - \min\{e,m\} = \mathbf{1}\{e<m\}$ for any integer cap $m$ (the min rises one step until $e$ reaches $m$, then is flat), applied at the two caps $m = v(a+1)$ and $m = va$ produced by the actions $a+1$ and $a$, giving, with $e := x+s$,
\[
\Delta_s \Delta_a\, r \;=\; q\bigl[\mathbf{1}\{e<v(a+1)\} - \mathbf{1}\{e<va\}\bigr] \;=\; q\,\mathbf{1}\{va \le e < v(a+1)\} \;\ge\; 0
\]
by (A1), the two indicators differing only on the capacity band $va \le e < v(a+1)$ that the added driver serves. For the continuation piece, recall the next backlog is $s' = (x+s-va)^+ = \max\{y,0\}$ with $y := x+s-va$, so the continuation enters $Q_t$ as $G(y)$. Routing it through the single point $y$ absorbs the clamp into $G$ and turns the two-variable dependence into shifts of one function, since raising $s$ by one sends $y \mapsto y+1$ while raising $a$ by one sends $y \mapsto y-v$ (the extra driver adds $v$ to the capacity $va$). Evaluating $G$ at the four corners $(s,a)$, $(s{+}1,a)$, $(s,a{+}1)$, $(s{+}1,a{+}1)$, which sit at $y$, $y{+}1$, $y{-}v$, $y{-}v{+}1$, the cross-difference is
\[
\Delta_s \Delta_a\, G(x+s-va) \;=\; \bigl[G(y) - G(y-v)\bigr] - \bigl[G(y+1) - G(y-v+1)\bigr] \;\ge\; 0,
\]
which is the $v$-step difference $z \mapsto G(z)-G(z-v)$ evaluated at $z=y$ minus its value at $z=y+1$. It is $\ge 0$ precisely because that $v$-step difference is non-increasing, the property Lemma~\ref{lem:continuation_vconcave} supplies.
Both inequalities hold pointwise in $x$, so taking the expectation over $x$ gives $\Delta_s \Delta_a Q_t \ge 0$; that is, $Q_t$ has increasing differences in $(s,a)$. By \textit{Topkis's theorem} \citep{topkis1998supermodularity} the smallest maximizer $a^*_t(s)$ is non-decreasing in $s$, which is~(ii).

We next prove the bounded increase~(iii). Write $\bar a := a^*_t(s)$. Any action with $va \ge x_{\max} + s$ already serves all demand and zeros next-period backlog, so a larger action only adds wage ($c_a > 0$); under the smallest-argmax convention $\bar a \le \lceil (x_{\max} + s)/v \rceil < a_{\max}$ by (A4), so $\bar a \in [a_{\min}, a_{\max}-1]$ and Lemma~\ref{lem:translation} applies. Subtracting the identity at $\bar a$ from the identity at any $a \in [a_{\min}, a_{\max}-1]$, where the two copies of $\kappa$ from \eqref{eq:translation} cancel,
\[
Q_t(s+v,\, a+1) - Q_t(s+v,\, \bar a + 1) \;=\; Q_t(s,\,a) - Q_t(s,\,\bar a) \;\le\; 0,
\]
the last inequality being optimality of $\bar a$ at $s$. Hence $\bar a + 1$ is the smallest maximizer of $Q_t(s+v,\cdot)$ over $[a_{\min}+1, a_{\max}]$, the image of $[a_{\min}, a_{\max}-1]$ under $a \mapsto a+1$. The only action outside that range is $a_{\min}$, leaving two cases:
\begin{itemize}
\item if $\bar a \ge a_{\min}+1$, then (ii) gives $a^*_t(s+v) \ge \bar a > a_{\min}$, so $a_{\min}$ is not the maximizer and $a^*_t(s+v) = \bar a + 1$;
\item if $\bar a = a_{\min}$, then $a^*_t(s+v)$ is either $a_{\min}$ or $a_{\min}+1$, both at most $\bar a + 1$.
\end{itemize}
In either case $\bar a \le a^*_t(s+v) \le \bar a + 1$, which is~(iii).

For non-increasingness and the strict decrease in~(i), evaluate the value at $s$ using the action $a^* := a^*_t(s+1)$ that is optimal one step up, which is feasible but generally suboptimal at $s$. Then optimality at $s$ gives $V_t(s) \ge Q_t(s,\, a^*_t(s+1))$, while $V_t(s+1) = Q_t(s+1, a^*)$ by definition. Subtracting keeps the action fixed across the two backlogs, so
\[
V_t(s+1) - V_t(s) \;\le\; Q_t(s+1, a^*) - Q_t(s, a^*) \;=\; \mathbb{E}_x\!\bigl[\Delta_s r + \Delta_s G(x+s-va)\bigr]\big|_{a = a^*_t(s+1)}.
\]
For every $(x,a)$, $\Delta_s r = q\,\mathbf{1}\{e<va\} - c_b \le q - c_b < 0$ by (A2), and $\Delta_s G(x+s-va) \le 0$ since $G$ is non-increasing; $V_t(s+1) < V_t(s)$ strictly.

It remains to prove the second half of $(H_t)$: that the $v$-step difference $M(s) := V_t(s+v) - V_t(s)$ is non-increasing. The point is that $M$ follows one of two formulas according to whether the optimal action at $s+v$ has left its floor $a_{\min}$, and Lemma~\ref{lem:threshold} stitches the two regimes together. We assemble its three ingredients.

First, the translation identity collapses the maximization at $s+v$. By Lemma~\ref{lem:translation}, each action $a \in [a_{\min}, a_{\max}-1]$ at $s$ corresponds to the action $a+1$ at $s+v$ with value raised by $\kappa$. Maximizing over this shifted block $[a_{\min}+1, a_{\max}]$ and pulling the constant $\kappa$ out of the maximum,
\[
\max_{a' \in [a_{\min}+1,\, a_{\max}]} Q_t(s+v, a') \;=\; \kappa + \max_{a \in [a_{\min},\, a_{\max}-1]} Q_t(s, a) \;=\; V_t(s) + \kappa,
\]
where the last equality uses that the maximizer $a^*_t(s)$ lies in $[a_{\min}, a_{\max}-1]$ by~(iii) and~(A4), so dropping the unused top action $a_{\max}$ from the maximum leaves the full value $V_t(s)$. The only action at $s+v$ outside the block is $a_{\min}$, the lone member of the set difference $[a_{\min}, a_{\max}] \setminus [a_{\min}+1, a_{\max}]$, so
\begin{equation}
V_t(s+v) \;=\; \max\bigl\{Q_t(s+v,\, a_{\min}),\; V_t(s) + \kappa\bigr\}.
\label{eq:Vtsv_max}
\end{equation}
Second, let $N(s) := Q_t(s+v,\, a_{\min}) - Q_t(s,\, a_{\min})$ be the $v$-step difference of $V_t$ in the regime where the action stays pinned at $a_{\min}$. Then $N$ is non-increasing: its reward part is non-increasing in $s$ because $s \mapsto \min\{x+s,\, va_{\min}\}$ is concave (slope $1$ until $x+s$ reaches the cap $va_{\min}$, then $0$), and the $v$-step difference of a concave function is non-increasing, and its continuation part is non-increasing by Lemma~\ref{lem:continuation_vconcave} (it is the $v$-step difference of the clamped continuation $G$).

Third, by~(ii) the action $a^*_t(s+v)$ is non-decreasing in $s$, so there is a threshold $s^\circ \in \mathbb{Z}_{\ge 0} \cup \{+\infty\}$ with $a^*_t(s+v) = a_{\min}$ exactly when $s < s^\circ$. This splits $M(s) = V_t(s+v) - V_t(s)$ into the two regimes Lemma~\ref{lem:threshold} expects, a floor regime where the optimal action at $s+v$ is pinned at $a_{\min}$ and a lifted regime where it has moved above $a_{\min}$:
\begin{itemize}
\item for $s < s^\circ$, (ii) also forces $a^*_t(s) = a_{\min}$, so both $V_t(s+v)$ and $V_t(s)$ are attained at $a_{\min}$ and $M(s) = N(s)$;
\item for $s \ge s^\circ$, the smallest-argmax convention gives $Q_t(s+v,\, a_{\min}) < V_t(s+v)$, so the second branch of \eqref{eq:Vtsv_max} is active and $M(s) = \kappa$.
\end{itemize}
When $0 < s^\circ < \infty$ we check the seam: the action $a_{\min}$ is still optimal at state $s^\circ - 1 + v$ (because $s^\circ - 1 < s^\circ$), so it beats $a_{\min}+1$ there, and rewriting the loser with Lemma~\ref{lem:translation} at $a = a_{\min}$,
\[
Q_t(s^\circ - 1 + v,\, a_{\min}) \;\ge\; Q_t(s^\circ - 1 + v,\, a_{\min}+1) \;=\; Q_t(s^\circ - 1,\, a_{\min}) + \kappa,
\]
which is exactly $N(s^\circ - 1) \ge \kappa$. The triple $(M, N, s^\circ)$, with $M(s) = V_t(s+v) - V_t(s)$ and $N(s) = Q_t(s+v, a_{\min}) - Q_t(s, a_{\min})$, now meets the hypotheses of Lemma~\ref{lem:threshold}, so $M$ is non-increasing. This establishes $(H_t)$, completing the inductive step and closing the backward induction begun above.
\end{proof}

\begin{lemma}
\label{lem:translation_aug}
For the Lagrangian-augmented reward with per-driver cost $c_a+\lambda$ ($\lambda\ge0$), Lemma~\ref{lem:translation} holds verbatim with $\kappa$ replaced by $\kappa_\lambda := \kappa-\lambda<0$.
\end{lemma}

\begin{restatable}{lemma}{lemReduction}
\label{lem:reduction}
If for every $\bvec{b}\in\Delta^{K-1}$ and every $\lambda \ge 0$ the function $D_{t+1}(s, \bvec{b}; \lambda)$ is non-decreasing in $s$ and satisfies $D_{t+1}(s + v, \bvec{b}; \lambda) - D_{t+1}(s, \bvec{b}; \lambda) \leq 1$, then the $Q$-function $Q_t$ exhibits decreasing differences in $(a, \lambda)$, and the optimal action $a^*(\lambda)$ is non-increasing in $\lambda$.
\end{restatable}

\begin{proof}[Proof of Lemma~\ref{lem:reduction}]
The argument has two parts: first, $V_t$ is convex and piecewise linear in $\lambda$ with finitely many breakpoints; then, off those breakpoints $\partial_\lambda Q_t$ has decreasing differences in $a$, which gives the monotone comparative statics via Topkis. Fix $\bvec{b}$; by (A3) the next belief $\bvec{b}'(x)$ is action-independent, so the inductive monotonicity hypothesis on $D_{t+1}$ holds at $\bvec{b}$ and at $\bvec{b}'(x)$ pointwise in $x$. By backward induction, $V_{t+1}(s, \bvec{b}; \cdot)$ is convex and piecewise linear in $\lambda$: $V_{T+1} = -c_{\mathrm{lost}}\,s$ is constant in $\lambda$, and if $V_{t+2}$ is convex and piecewise linear in $\lambda$ then $Q_{t+1}(s, \bvec{b}, a; \lambda) = \mathbb{E}_x[r(x,s,a) - \lambda a + V_{t+2}(s', \bvec{b}'; \lambda)]$ is affine plus convex piecewise linear at fixed $a$, and $V_{t+1} = \max_a Q_{t+1}$ inherits the structure as a pointwise maximum over the finitely many actions $a\in\mathcal{A}$ of convex piecewise-linear functions of $\lambda$ (a max of convex PWL functions is convex PWL). The breakpoints in $\lambda$ form a finite set at every $(s, \bvec{b})$.

We now turn to the second part. Off the breakpoint set just identified the optimal action $a^*(\lambda)$ is locally constant, so by the \textit{envelope theorem} \citep{MilgromSegal2002} applied to $V_t=\max_a Q_t$ we have $\partial V_t/\partial\lambda=\partial_\lambda Q_t(s,\bvec{b},a^*;\lambda)$, which equals $-D_t(s,\bvec{b};\lambda)$ by evaluating the per-stage identity below at $a=a^*$ and unrolling the recursion. Differentiating $Q_t$ in $\lambda$ at fixed $a$ and applying the envelope identity to the continuation,
\[
\frac{\partial Q_t}{\partial \lambda}(s, \bvec{b}, a; \lambda) \;=\; -a \,-\, \mathbb{E}_x\bigl[D_{t+1}(s'(s,a), \bvec{b}'(x); \lambda)\bigr], \qquad s'(s,a) := \max\{s+x-va, 0\}.
\]
Evaluating at $a = a^*$ and using the one-step recursion $D_t(s) = a^* + \mathbb{E}_x[D_{t+1}(s'(s,a^*))]$, the right side is $-(a^* + \mathbb{E}_x[D_{t+1}]) = -D_t(s, \bvec{b}; \lambda)$, which is the claimed identity $\partial_\lambda V_t = -D_t$.

This identity holds by backward induction on $t$: the base case $V_{T+1} = -c_{\mathrm{lost}}\,s$ is independent of $\lambda$, so $\partial_\lambda V_{T+1} = 0 = -D_{T+1}$, and the continuation term above carries the inductive hypothesis $\partial_\lambda V_{t+1} = -D_{t+1}$. The derivative of $Q_t$ at a fixed action $a$ is read directly from its explicit form, in which only the wage term $-\lambda a$ and the continuation depend on $\lambda$; the envelope theorem is used only for the value $V_t = \max_a Q_t$ at its optimum.

Taking the discrete difference in $a$,
\[
\Delta_a\!\bigl[\partial_\lambda Q_t\bigr] \;=\; -1 \,+\, \mathbb{E}_x\bigl[D_{t+1}(s'(s,a), \bvec{b}'; \lambda) - D_{t+1}(s'(s,a+1), \bvec{b}'; \lambda)\bigr].
\]
Setting $y := s+x-va$, the successor gap $s'(s,a) - s'(s,a+1) = \max\{y,0\} - \max\{y-v,0\}$ lies in $[0, v]$ ($\max\{\cdot,0\}$ is non-decreasing and $1$-Lipschitz), with the extremes $0$ at $y < 0$ and $v$ at $y \ge v$. Inductive monotonicity of $D_{t+1}$ together with the bounded-$v$-step hypothesis bounds the bracketed difference by $1$, so $\Delta_a[\partial_\lambda Q_t]\le0$ wherever differentiable. Each $Q_t(s,\bvec{b},a;\cdot)$ is piecewise linear, hence continuous, in $\lambda$, so $\Delta_a Q_t$ is continuous; a continuous function that is non-increasing off the finite set of breakpoints is non-increasing on all of $[0,\infty)$. Hence $Q_t$ has decreasing differences in $(a,\lambda)$. Therefore, by \textit{Topkis's theorem} \citep{topkis1998supermodularity}, $a^*(\lambda)$ is non-increasing in $\lambda$.
\end{proof}

\thmIndexability*
\begin{proof}[Proof of Theorem~\ref{prop:indexability}]
The goal is to verify the two hypotheses of Lemma~\ref{lem:reduction}, labeled (a)-(b) below, by backward induction, after which the lemma delivers indexability; the Lagrangian augmentation only shifts the per-driver cost, so parts~(ii)--(iii) of Theorem~\ref{prop:backlog_monotone} apply verbatim. Fix $\bvec{b}$ and $\lambda \ge 0$ and write $D_t(s) := D_t(s, \bvec{b}; \lambda)$ and $a^*_t(s) := a^*_t(s, \bvec{b}; \lambda)$ under the smallest-argmax convention; by (A3) every expectation below is taken at the same next belief $\bvec{b}'(x)$ pointwise in $x$. Recall that $D_t(s)$ is the expected driver-hours the optimal policy spends from $(s,t)$ to the horizon, equivalently the marginal value $-\partial_\lambda V_t$ of relaxing the driver budget (the two readings coincide by the envelope identity $\partial_\lambda V_t = -D_t$ derived in the proof of Lemma~\ref{lem:reduction}; raising the price $\lambda$ lowers value at a rate equal to the driver-hours bought). By Lemma~\ref{lem:reduction}, indexability holds once $D_t$ is non-decreasing in $s$ and grows by at most one over a step of $v$: an extra unit of backlog never lets the store return driver-hours, and $v$ extra units cost at most one more driver-hour, since one driver clears $v$ orders. We prove both by backward induction on $t$: for every $s \in \mathbb{Z}_{\ge 0}$ and every $\lambda \ge 0$,
\begin{enumerate}[label=(\alph*)]
  \item $D_t(s+1, \bvec{b}; \lambda) \ge D_t(s, \bvec{b}; \lambda)$,
  \item $D_t(s+v, \bvec{b}; \lambda) - D_t(s, \bvec{b}; \lambda) \le 1$.
\end{enumerate}
The terminal case $D_{T+1} \equiv 0$ satisfies both, which is the base case $t = T+1$. For the inductive step, assume (a)--(b) at $t+1$ and establish them at $t$.

We first import the policy structure and record two facts. By Lemma~\ref{lem:translation_aug}, the augmented reward (per-driver cost $c_a + \lambda$) obeys the translation identity with constant $\kappa_\lambda = \kappa - \lambda < 0$, and (A1)--(A4) do not involve $\lambda$, so parts~(ii) and~(iii) of Theorem~\ref{prop:backlog_monotone} apply to $a^*_t(\cdot, \bvec{b}; \lambda)$: it is non-decreasing in $s$ with $a^*_t(s+v) \le a^*_t(s) + 1$. Write $\bar{a} := a^*_t(s)$ and $s'_\sigma := \max\{\sigma + x - v\,a^*_t(\sigma),\, 0\}$ for the next-period backlog from a generic starting backlog $\sigma$ (we use $\sigma$ rather than $r$ here so as not to clash with the reward $r(x,s,a)$). The driver-hours obey the one-step recursion $D_t(\sigma) = a^*_t(\sigma) + \mathbb{E}_x[D_{t+1}(s'_\sigma)]$, the drivers hired now plus the expected driver-hours from the successor onward. Writing it at $\sigma$ and at $s$ and subtracting, and combining the two expectations under a common demand $x$ (legitimate because (A3) makes the next belief $\bvec{b}'(x)$ action-independent, so both successors sit at the same belief), the cost-to-go splits into an action term and a continuation term,
\begin{equation}
D_t(\sigma) - D_t(s) \;=\; \bigl[a^*_t(\sigma) - \bar a\bigr] \;+\; \mathbb{E}_x\bigl[D_{t+1}(s'_{\sigma}) - D_{t+1}(s'_s)\bigr], \qquad \sigma \in \{s+1,\, s+v\}.
\label{eq:D_decomp}
\end{equation}
Moreover, the hypotheses (a)--(b) at $t+1$ make $D_{t+1}$ non-decreasing with $v$-step at most one, so for any two successors with $0 \le s' - s'' \le v$,
\begin{equation}
0 \;\le\; D_{t+1}(s') - D_{t+1}(s'') \;\le\; 1.
\label{eq:Dcont_bound}
\end{equation}

We first verify~(a), taking $\sigma = s+1$ in \eqref{eq:D_decomp}. By~(ii)--(iii), $a_1 := a^*_t(s+1) \in \{\bar a, \bar a+1\}$. With $y := s + x - v\bar a$ the two successors are $s'_s = \max\{y, 0\}$ and $s'_{s+1} = \max\{y + 1 - v(a_1 - \bar a),\, 0\}$ (the latter from $s'_{s+1} = \max\{(s+1)+x - v a_1,\, 0\}$ after substituting $y$), and we bound \eqref{eq:D_decomp} in each case:
\begin{itemize}
\item if $a_1 = \bar a$, the action term is $0$ and $s'_{s+1} - s'_s \in \{0,1\}$, so the continuation term is $\ge 0$ by \eqref{eq:Dcont_bound}; hence $D_t(s+1) - D_t(s) \ge 0$;
\item if $a_1 = \bar a + 1$, the action term is $1$ and $s'_{s+1} = \max\{y - v + 1,\, 0\} \le s'_s$ with $s'_s - s'_{s+1} \le v$, so the continuation term is $\ge -1$ by \eqref{eq:Dcont_bound}; hence $D_t(s+1) - D_t(s) \ge 1 - 1 = 0$.
\end{itemize}

We next verify~(b), taking $\sigma = s+v$ in \eqref{eq:D_decomp}. Let $\delta := \bar a + 1 - a^*_t(s+v) \in \{0,1\}$, which by~(iii) measures how far the action at $s+v$ falls short of $\bar a + 1$. The successors are $s'_s = \max\{s + x - v\bar a,\, 0\}$ and $s'_{s+v} = \max\{s + x - v\bar a + v\delta,\, 0\}$ (using $a^*_t(s+v) = \bar a + 1 - \delta$, so the start $s+v$ gives $\max\{(s+v)+x - v\,a^*_t(s+v),\, 0\}$), so $0 \le s'_{s+v} - s'_s \le v\delta$. The action term is $a^*_t(s+v) - \bar a = 1 - \delta$, and by \eqref{eq:Dcont_bound} the continuation term is at most $\delta$: when $\delta = 1$ the successors differ by at most $v$, giving at most $1$; when $\delta = 0$ they coincide, giving $0$. Hence
\[
D_t(s+v) - D_t(s) \;\le\; (1 - \delta) + \delta \;=\; 1.
\]
The induction closes, and Lemma~\ref{lem:reduction} yields $a^*_t(s, \bvec{b}; \lambda)$ non-increasing in $\lambda$ at every state; that is, the single-store POMDP is indexable.
\end{proof}

\corIndexMono*
\begin{proof}[Proof of Corollary~\ref{cor:index_monotone}]
Fix a store and drop the subscript $\ell$, fix $t$ and $\bvec{b}$, and write $a^*(\lambda; s)$ for the smallest optimal action at state $(s, \bvec{b})$ under shadow price $\lambda$, so that $\Lambda^{(j)}(s, \bvec{b}) = \inf\{\lambda \ge 0 : a^*(\lambda; s) < j\}$ by Equation~\eqref{eq:index}. By Lemma~\ref{lem:translation_aug} the augmented problem satisfies the translation identity with constant $\kappa_\lambda < 0$, and (A1)--(A4) do not involve $\lambda$, so parts~(ii) and~(iii) of Theorem~\ref{prop:backlog_monotone} apply to $a^*(\lambda; \cdot)$ at every price $\lambda \ge 0$, exactly as in the proof of Theorem~\ref{prop:indexability}.

We first show the index is non-decreasing in $s$. Fix backlogs $s_1 \le s_2$ and a rank $j$. For any $\lambda$ with $a^*(\lambda; s_2) < j$, part~(ii) gives $a^*(\lambda; s_1) \le a^*(\lambda; s_2) < j$, so every price at which the store demands fewer than $j$ drivers at backlog $s_2$ does so at backlog $s_1$ as well. The set $\{\lambda \ge 0 : a^*(\lambda; s_2) < j\}$ is therefore contained in $\{\lambda \ge 0 : a^*(\lambda; s_1) < j\}$, and the infimum over the smaller set is no smaller. Hence $\Lambda^{(j)}(s_1, \bvec{b}) \le \Lambda^{(j)}(s_2, \bvec{b})$.

It remains to prove the second inequality. For any $\lambda$ with $a^*(\lambda; s) < j$, part~(iii) gives $a^*(\lambda; s+v) \le a^*(\lambda; s) + 1 < j+1$, so $\{\lambda \ge 0 : a^*(\lambda; s) < j\}$ is contained in $\{\lambda \ge 0 : a^*(\lambda; s+v) < j+1\}$. Taking infima, and noting the infimum over the larger set is no larger, $\Lambda^{(j+1)}(s+v, \bvec{b}) \le \Lambda^{(j)}(s, \bvec{b})$.
\end{proof}

% CITY ORDER (alphabetical, with integer `city` code): ALC=4, BCN=0, BIL=10, GRA=7, LPA=9 (dropped), MAD=1, MAL=6, PAL=5, PNA=2, SEV=3, TFN=11 (dropped), VAL=12, ZAR=8
% City dict (integer `city` code -> anonymized letter), SOURCE OF TRUTH:
%   {0:'E', 1:'B', 2:'I', 3:'H', 4:'A', 5:'G', 6:'F', 7:'D', 8:'K', 9:'L', 10:'C', 11:'M', 12:'J'}
%   string codes: BCN=E, MAD=B, PNA=I, SEV=H, ALC=A, PAL=G, MAL=F, GRA=D, ZAR=K, LPA=L, BIL=C, TFN=M, VAL=J
%   (keys 9=LPA and 11=TFN were previously dropped; now assigned L and M)
% Store dict ((city_code, cat_mfc) -> anonymized store label); single-store cities use the bare letter:
%   B (MAD): 3->B1, 4->B2, 5->B3, 18->B4, 21->B5, 22->B6, 23->B7, 26->B8   [focal store MAD-22 = B6]
%   E (BCN): 0->E1, 1->E2, 2->E3, 6->E4, 7->E5, 17->E6, 19->E7, 24->E8
%   single-store: A=ALC-10, C=BIL-16, D=GRA-13, F=MAL-12, G=PAL-11, H=SEV-9, I=PNA-8, J=VAL-25, K=ZAR-14, L=LPA-15, M=TFN-20

\section{Data Descriptive Table}~\begin{table}[h!]
  % \scriptsize
    \centering
    \begin{tabular}{cccccccc}
    \toprule
    City   & No. Stores & Mean   & St.Dev.   & Max.   & Length (days) & KPSS Statistic$^\dagger$ & Ap.Entropy$^\ddagger$ \\
    \midrule
    A & 1 & 4.63  & 3.41  & 18  & 296 & 8.21   & 2.00 \\
    B & 8 & 67.81 & 45.95 & 501 & 542 & 9.62   & 1.23 \\
    C & 1 & 4.36  & 3.31  & 23  & 352 & 2.36   & 2.01 \\
    D & 1 & 6.36  & 3.84  & 19  & 142 & 0.41*  & 1.77 \\
    E & 8 & 69.36 & 47.74 & 280 & 575 & 18.36  & 1.24 \\
    F & 1 & 7.84  & 5.18  & 30  & 275 & 5.70   & 1.55 \\
    G & 1 & 10.16 & 6.00  & 36  & 226 & 2.66   & 1.69 \\
    H & 1 & 8.63  & 5.33  & 31  & 205 & 2.49   & 1.58 \\
    I & 1 & 5.33  & 4.06  & 25  & 321 & 2.60   & 1.96 \\
    J & 1 & 14.69 & 9.15  & 56  & 225 & 4.82   & 1.79 \\
    K & 1 & 7.25  & 5.09  & 31  & 345 & 1.59   & 1.52 \\
    L & 1 & 4.13  & 3.29  & 23  & 342 & 10.15  & 1.98 \\
    M & 1 & 5.97  & 3.95  & 25  & 323 & 6.40   & 2.02 \\
    \bottomrule
    \multicolumn{8}{l}{$\dagger$: Kwiatkowski–Phillips–Schmidt–Shin (KPSS) test statistic for stationarity.}\\ 
    \multicolumn{8}{l}{ \quad * stationarity hypothesis not rejected ($p>.05$).}\\
    \multicolumn{8}{l}{$\ddagger$: Approximate Entropy to measure time series irregularity.}
    \end{tabular}%
\caption{Descriptive statistics of hourly demand by city. Statistics are computed on each city's hourly demand summed across stores over operating hours, with Approximate Entropy evaluated at $m=2$, $r=0.2\sigma$.}
  \label{table:descriptive}%
\end{table}%

\section{Information Ladder for the Bayesian Filter}
\label{app:info_ladder}

Figure~\ref{fig:info_ladder} places the realized belief-updating gain against the relevant information bounds, using synthetic days from the fitted model for store~B6 so the true regime path is known. Because these days are generated under the fitted model with the regime path known, the ladder measures what the filter can extract within the model. Perfect information about the current regime, the expected value of perfect information (EVPI), would be worth about 20\% per day over the static-belief policy, but no causal policy can reach it as the latent state cannot be directly observed in practice. Knowing the previous hour's regime, propagated one step through $\bvec{T}$, is worth about 12\%, so approximately half of the EVPI is lost to one hour of regime mixing. Our Bayesian filter captures about 8\%, two thirds of this attainable bound, and the shortfall reflects regime emissions too overlapping to pin down the regime from a single hour of demand.~\begin{figure}[h!]
    \centering
    \includegraphics[width=0.35\textwidth]{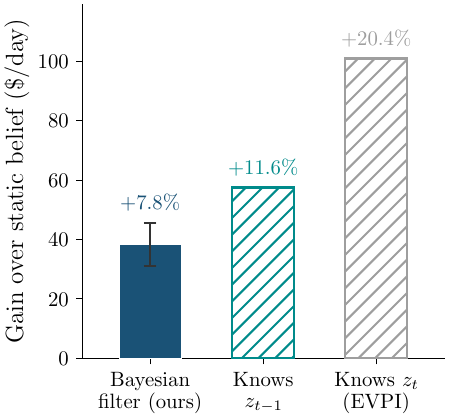}
    \caption{Mean daily gain over the static-belief policy of the different information gates (whisker is the 95\% confidence interval).}
    \label{fig:info_ladder}
\end{figure}

\section{Solution Algorithms}
\label{app:algorithms}

This appendix provides the full pseudocode for the backward induction, Q-function estimation, and online execution procedures described in Section \ref{sec:algorithm}.

\begin{algorithm}[H]
\caption{Finite-Horizon Backward Induction (per day-of-week $d$)}
\label{alg:backward_induction}
\begin{algorithmic}[1]
\Require Regime params $\{(\theta_k, \sigma_k)\}_{k=1}^K$, baselines $\{\mu_{d,h}\}$, transition $\bvec{T}$, cost parameters, operating-hour sequence $\mathcal{H}_d = (h_1, \ldots, h_T)$ for each day-of-week $d$
\Ensure Day- and time-indexed policy $\pi(s, \bvec{b}, t, d)$ for $t = 1, \ldots, T$
\State Discretize backlog: $\mathcal{S} = \{0, 1, \ldots, s_{\max}\}$
\State Discretize belief simplex: $\mathcal{B} \subset \Delta^{K-1}$, regular grid with step $\Delta b = 0.05$
    \Comment{$|\mathcal{B}| = \binom{K - 1 + 1/\Delta b}{K - 1}$}
\State Discretize actions: $\mathcal{A} = \{a_{\min}, \ldots, a_{\max}\}$
\For{each day-of-week $d = 1, \ldots, 7$}
    \State Set terminal value: $V_{T+1}(s, \bvec{b};\, d) \gets -c_{\mathrm{lost}} \cdot s$ for all $(s, \bvec{b})$
    \For{$t = T, T-1, \ldots, 1$}
        \For{each $(s, \bvec{b}) \in \mathcal{S} \times \mathcal{B}$}
            \For{each $a \in \mathcal{A}$}
                \State $Q(s, \bvec{b}, a) \gets \Call{EstimateQ$_t$}{s, \bvec{b}, a, V_{t+1}(\cdot;\, d), d, h_t}$
            \EndFor
            \State $V_t(s, \bvec{b};\, d) \gets \max_{a \in \mathcal{A}}\; Q(s, \bvec{b}, a)$
            \State $\pi(s, \bvec{b}, t, d) \gets \arg\max_{a \in \mathcal{A}}\; Q(s, \bvec{b}, a)$
        \EndFor
    \EndFor
\EndFor
\end{algorithmic}
\end{algorithm}

\begin{algorithm}[H]
\caption{\textsc{EstimateQ$_t$}: Monte Carlo Q-value with known calendar context}
\label{alg:estimate_q_t}
\begin{algorithmic}[1]
\Procedure{EstimateQ$_t$}{$s, \bvec{b}, a, V, d, h_t$}
\State $Q \gets 0$
\For{each regime $k = 1, \ldots, K$}
    \State $R_k \gets 0$
    \For{$m = 1, \ldots, N_{\text{mc}}$}
        \State Sample $x \sim f_k(\cdot;\, d, h_t)$
            \Comment{Known $(d, h_t)$; no uniform sampling}
        \State $r \gets r(x, s, a)$ \Comment{Eq.~\ref{eq:reward_spillover}}
        \State $s' \gets \max\{x + s - v \cdot a,\; 0\}$
        \State $\bvec{b}' \gets \textsc{BayesUpdate}(\bvec{b},\, \bvec{T},\, x,\, d,\, h_t)$
            \Comment{Eq.~\ref{eq:belief_combined}}
        \State $R_k \mathrel{+}= r + V(s', \bvec{b}')$
            \Comment{No discount factor}
    \EndFor
    \State $Q \mathrel{+}= b(k) \cdot R_k / N_{\text{mc}}$
\EndFor
\State \Return $Q$
\EndProcedure
\end{algorithmic}
\end{algorithm}

Since demand varies by day of week, a separate backward induction is solved for each $d = 1, \ldots, 7$, yielding seven policies per store. The outer loop over $d$ is parallelizable.

After \textsc{BayesUpdate}, the updated belief $\bvec{b}'$ will generally not lie on a grid point. The value lookup $V(s', \bvec{b}')$ uses barycentric interpolation on the simplex. We identify the enclosing simplex cell in the regular grid and compute $V$ as a convex combination of the cell's $K$ vertex values. For a regular grid on $\Delta^{K-1}$, the enclosing cell and barycentric weights can be found in $O(K \log K)$ time.

\begin{algorithm}[H]
\caption{Online Execution of the POMDP Policy}
\label{alg:online}
\begin{algorithmic}[1]
\Require Policy $\pi(\cdot,\, \cdot,\, t, d)$, transition $\bvec{T}$, regime params $\{(\theta_k, \sigma_k)\}_{k=1}^K$, baselines $\{\mu_{d,h}\}$, stationary distribution $\bvec{\pi}$
\For{each operating day with day-of-week $d$}
    \State Initialize backlog $s_1 \gets 0$ (or carry over from previous day)
    \State Initialize belief $\bvec{b}_1 \gets \bvec{\pi}$ \Comment{Stationary distribution of $\bvec{T}$; Section~\ref{sec:transition}}
    \For{$t = 1, 2, \ldots, T$}
        \State \textbf{Decide:} $a_t \gets \pi(s_t, \bvec{b}_t, t, d)$
        \State \textbf{Observe:} demand $x_t$ realizes at hour $h_t$
        \State \textbf{Collect reward:} $r_t \gets r(x_t, s_t, a_t)$
        \State \textbf{Update backlog:} $s_{t+1} \gets \max\{x_t + s_t - v \cdot a_t,\; 0\}$
        \State \textbf{Correct belief:}
            $b_t^+(j) \gets \dfrac{b_t(j) \cdot f_j(x_t;\, d, h_t)}
                                  {\sum_{k} b_t(k) \cdot f_k(x_t;\, d, h_t)}$
            for all $j$ \Comment{Eq.~\ref{eq:correct}}
        \State \textbf{Predict belief:}
            $\bvec{b}_{t+1} \gets \bvec{T}^\top \bvec{b}_t^+$
            
    \EndFor
\EndFor
\end{algorithmic}
\end{algorithm}

The initial belief $\bvec{b}_1 = \bvec{\pi}$ is the stationary distribution of $\bvec{T}$, computed as the normalized left eigenvector corresponding to eigenvalue 1. This is the principled uninformative prior for the first operating hour of a session.

\end{document}